\documentclass[preprint,12pt]{elsarticle}
\usepackage[utf8]{inputenc}
\usepackage[T1]{fontenc}
\usepackage[a4paper]{geometry}
\usepackage{amsmath,amssymb}
\usepackage{textcomp}
\usepackage{enumerate}
\usepackage{tikz}
\usetikzlibrary{arrows,decorations.pathmorphing}
\usetikzlibrary{backgrounds,positioning,fit,petri}
\usetikzlibrary{matrix}
\newenvironment{proof}{\par \noindent {\sc Proof}\ --\ }{$\square$ \par}
\def\A{\mathcal{A}}
\def\QX{\Q\langle X \rangle}

\def\1{\mathbb{1}}

\def\scal#1#2{\langle #1\bv#2 \rangle}
\def\bv{\mid}

\def\abs#1{\bv\!#1\!\bv}

\def\ncp#1#2{#1\langle #2\rangle}

\newcounter{per1}
\begingroup
\def\2#1{\ifnum#1<10 0\fi\the#1}
\xdef\isodayandtime{
{\2\day-\2\month-\the\year\space\2{\count0}:%
\2{\count2}}}
\endgroup

\newcommand{\bi}{\begin{itemize}}
\newcommand{\ei}{\end{itemize}}
\newcommand{\bd}{\begin{description}}
\newcommand{\ed}{\end{description}}

\usepackage{amsmath}
\usepackage{amsfonts}
\usepackage{amssymb}

\newcommand{\calH}{{\mathcal H}}
\newcommand{\calC}{{\mathcal C}}

\newcommand{\N}{{\mathbb N}}

\newcommand{\Z}{{\mathbb Z}}
\newcommand{\Q}{{\mathbb Q}}
\newcommand{\R}{{\mathbb R}}
\newcommand{\C}{{\mathbb C}}
\newcommand{\K}{{\mathbb K}}

\newcommand{\Frac}[2]{\displaystyle \frac{#1}{#2}}
\newcommand{\Sum}[2]{\displaystyle{\sum_{#1}^{#2}}}
\newcommand{\Prod}[2]{\displaystyle{\prod_{#1}^{#2}}}

\newcommand{\Lim}[1]{\displaystyle{\lim_{#1}\ }}

\def\Lyn{{\mathcal Lyn}}

\def\abs#1{|#1|}

 \def\shuffle{\mathop{_{^{\sqcup\!\sqcup}}}} 

\gdef\stuffle{\;%
  \setlength{\unitlength}{0.0125cm}%
  \begin{picture}(20,10)(220,580) 
  \thinlines 
  \put(220,592){\line( 0,-1){ 10}} 
  \put(220,582){\line( 1, 0){ 20}} 
  \put(240,582){\line( 0, 1){ 10}} 
  \put(230,592){\line( 0,-1){ 10}} 
  \put(225,587){\line( 1, 0){ 10}} 
  \end{picture}\; 
}

\newtheorem{corollary}{Corollary}
\newtheorem{proposition}{Proposition}
\newtheorem{theorem}{Theorem}
\newtheorem{lemma}{Lemma}
\newtheorem{definition}{Definition}

\newtheorem{example}{Example}

\newcommand{\Li}{\operatorname{Li}}

\def\L{\mathrm{L}}
\def\H{\mathrm{H}}

\def\A{\mathrm{A}}

\def\G{\mathrm{G}}

\def\deg{\mathrm{deg}}

\newcommand{\poly}[2]{#1 \langle #2 \rangle}

\def\QX{\poly{\Q}{X}}

\def\Lim{\displaystyle\lim}
\def\Sum{\displaystyle\sum}
\def\Prod{\displaystyle\prod}
\def\Frac{\displaystyle\frac}
\def\path{\rightsquigarrow}

\def\bv{\mid}

\def\abs#1{\bv\!#1\!\bv}

\gdef\minishuffle{{\scriptstyle \shuffle}}  
\gdef\ministuffle{{\scriptstyle \stuffle}}

\def\deg{\mathop\mathrm{deg}\nolimits}

\def\binom#1#2{{#1\choose#2}}

\usepackage{color}
\usepackage{ulem}

\def\scal#1#2{\langle #1\bv#2 \rangle}
\def\ncp#1#2{#1\langle #2\rangle}

\usepackage{amssymb}

\journal{Journal of  Symbolic Computation}

\begin{document}

\begin{frontmatter}

\title{\Large \textbf{\noindent
Harmonic sums and polylogarithms\\ at non-positive multi-indices}}
\author{\normalsize{
G\'erard H. E. Duchamp$^{\diamondsuit}$ - V. Hoang Ngoc Minh$^{\heartsuit}$ - Ngo Quoc Hoan$^{\blacklozenge}$
}}

\address{\textit{\footnotesize
$^{\diamondsuit}$ Paris XIII University, 93430 Villetaneuse, France,
gheduchamp@gmail.com\\
$^{\heartsuit}$ Lille II University, 59024 Lille, France,
hoang@univ-lille2.fr\\
$^{\blacklozenge}$ Paris XIII University, 93430 Villetaneuse, France,
quochoan\_ngo@yahoo.com.vn\\
}}

\begin{abstract}
Extending Eulerian polynomials and Faulhaber's formula\footnote{First seen and computed up to order 17 by Faulhaber.
The modern form and proof are credited to Bernoulli \cite{Knuth}.}, we study several combinatorial aspects of harmonic sums and polylogarithms
at non-positive multi-indices as well as their structure. Our techniques are based on the combinatorics of noncommutative generating series
in the shuffle Hopf algebras giving a global process to renormalize the divergent polyzetas at non-positive multi-indices.
\end{abstract}

\end{frontmatter}

\section{Introduction}

The story begins with the celebrated Euler sum \cite{euler1}
\begin{eqnarray*}
\zeta(s)=\sum_{n\geq 1}n^{-s},&&s\in\N,s>1.
\end{eqnarray*}
Euler gave an explicit formula expressing the following ratio (with $\mathrm{i}^2=-1$)~:
\begin{eqnarray}\label{pair}
\forall j\in\N_+,&\Frac{\zeta(2j)}{(2\mathrm{i}\pi)^{2j}}=-\Frac12\Frac{b_{2j}}{(2j)!}&\in\Q,
\end{eqnarray}
where $\{b_j\}_{j \in \N}$ are the Bernoulli numbers. Multiplying two such sums, he obtained
\begin{eqnarray*}
\zeta(s_1)\zeta(s_2)=\zeta(s_1,s_2)+\zeta(s_1+s_2)+\zeta(s_2,s_1),
\end{eqnarray*}
where the {\it polyzeta} are given by
\begin{eqnarray*}
\zeta(s_1,\ldots,s_r)=\sum_{n_1>\ldots>n_r>0}n_1^{-s_1}\ldots n_r^{-s_r},&&r,s_1,\ldots,s_r\in\N_+,s_1>1.
\end{eqnarray*}
Establishing relations among polyzetas, $\zeta(s_1,s_2)$ with $s_1+s_2\le16$, he proved \cite{euler2}
\begin{eqnarray}\label{double}
\forall s>1,&&\zeta(s,1)=\frac12\zeta(s+1)-\frac12\sum_{j=1}^{s-1}\zeta(j+1)\zeta(s-j).
\end{eqnarray}
Extending \eqref{double}, Nielsen showed that\footnote{Nielsen wished to establish
an identity analogous to the one given in (\ref{pair}) for $\zeta(2p+1)$ but he did not succeed.
We have explained in \cite{SLC43,FPSAC97} these difficulties and impossibility \cite{VJM}.}
$\zeta(s,\{1\}^{r-1})$ is an homogenous polynomial of degree $n+r$ of $\{\zeta(2),\ldots,\zeta(s+r)\}$
with rational coefficients \cite{nielsen1,nielsen2,nielsen3} (see also \cite{FPSAC95,FPSAC96}).

After that, Riemann extended $\zeta(s)$ as a meromorphic function \cite{riemann} on $\C$.
The series converges absolutely in $\calH_1=\{s\in\C\vert\Re(s)>1\}$.
Moreover, if $\Re(s)\ge a>1$, it is dominated, term by term, 
by the absolutely convergent series, of general term $n^{-a}$ so, by Cauchy's criterion, it converges in $\calH_1$ 
(compact uniform convergence and then represents a holomorphic fonction in\footnote{which is actually the domain 
of absolute convergence of this univariate series.} $\calH_1$ \cite{riemann}).

In the same vein $\zeta(s_1,\ldots,s_r)$ is well-defined on $\C^r$ as a meromorphic function \cite{Akiyama,Goncharov,Zhao,Matsumoto}.
Denoting $t_0=1,u_{r+1}=1$ and
\begin{eqnarray}\label{lambda}
\lambda(z):={z}/{(1-z)},
\end{eqnarray}
this can be done via the following integral representations
obtained by the convolution theorem and by changes of variables \cite{IMACS,FPSAC95,FPSAC96}
\begin{eqnarray*}
\zeta(s_1,\ldots,s_r)
&=&\int_0^1\frac{dt_1}{1-t_1}\frac{\log^{s_1-1}({t_0}/{t_1})}{\Gamma(s_1)}\ldots\int_0^{t_r-1}
\frac{dt_r}{1-t_r}\frac{\log^{s_r-1}({t_{r-1}}/{t_r})}{\Gamma(s_r)}\cr
&=&\int_{[0,1]^r}\prod_{j=1}^r\log^{s_j-1}(\frac1{u_j}) \frac{\lambda(u_1\ldots u_j)}{\Gamma(s_j)}\frac{du_j}{u_j}
=\int_{\R_+^r}\prod_{j=1}^r\frac{\lambda(e^{-(u_1\ldots u_j)})}{\Gamma(s_j)} \frac{du_j}{u_j^{1-s_j}}.
\end{eqnarray*}

In fact, one has two ways of thinking polyzetas as limits, fulfilling identities.
Firstly, they are limits of {\it polylogarithms}, at $z=1$, and secondly, as truncated sums,
they are limits of {\it harmonic sums} when the upper bound tends to $+\infty$.
The link between these holomorphic and arithmetic functions is as follows. 
 
For any $r$-uplet $(s_1,\ldots,s_r)\in\N_+^r$, $r\in\N_+$ and for any $z\in\C$ such that $\abs{z}<1$,
the polylogarithm and the harmonic sum are well defined by
\begin{eqnarray*}\label{polylogarithm}
\Li_{s_1,\ldots,s_r}(z):=\sum_{n_1>\ldots>n_r>0}\frac{z^{n_1}}{n_1^{s_1}\ldots n_r^{s_r}}&\mbox{and}&
\H_{s_1,\ldots,s_r}(N):=\sum_{N\ge n_1>\ldots>n_r>0}\frac1{n_1^{s_1}\ldots n_r^{s_r}}.
\end{eqnarray*}
These objects appeared within the functional expansions in order to represent the nonlinear dynamical systems
in quantum electrodynanics and have been developped by Tomonaga, Schwinger and Feynman \cite{dyson}.
They appeared then in the singular expansion of the solutions and their successive
(ordinary or functional) derivations \cite{fliess1} of nonlinear differential equations
with three singularities \cite{BBM,QED,cade,acta} and then they also appeared in the asymptotic expansion
of the Taylor coefficients. The main challenge of these expansions lies in the divergences
and leads to problems of {\it regularization} and {\it renormalization}
which can be solved by combinatorial technics \cite{CK,interface,QED,FH,fliess1,cade,acta,lemurakami}.

Let\footnote{For $m\ge2$, the domain of absolute convergence of $\zeta(s_1,\ldots,s_r)$
contains the domain $\calH_r$ \cite{Matsumoto}.}
$\calH_r=\{(s_1,\ldots,s_r)\in\C^r\vert\forall m=1,\ldots,r,\Re(s_1)+\ldots+\Re(s_m)>m\}$.
From the analytic continuation point of view \cite{Akiyama,jtnb,jtnb2,Goncharov,Matsumoto,Zhao} and after a theorem by Abel, one has
\begin{eqnarray*}
\forall (s_1,\ldots,s_r)\in\calH_r,&&
\zeta(s_1,\ldots,s_r)=\lim_{z \rightarrow 1}\Li_{s_1,\ldots,s_r}(z)=\lim_{N \rightarrow\infty}\H_{s_1,\ldots,s_r}(N).
\end{eqnarray*}
This theorem is no more valid in the divergent cases as, for
$(s_1,\ldots,s_r)\in\N^r$,
\begin{eqnarray*}
\Li_{\{1\}^k,s_{k+1},\ldots,s_r}(z)=&\Sum_{n_1>\ldots>n_r>0}\Frac{z^{n_1}}{n_1\ldots n_kn_{k+1}^{s_{k+1}}\ldots n_r^{s_r}},&\cr
\H_{\{1\}^k,s_{k+1},\ldots,s_r}(N)=&\Sum_{N\geq n_1>\ldots>n_r>0}\Frac1{n_1\ldots n_kn_{k+1}^{s_{k+1}}\ldots n_r^{s_r}},&\cr
\Li_{\{1\}^r}(z)=&\Sum_{n_1>\ldots>n_r>0}\Frac{z^{n_1}}{n_1\ldots n_r}&=\Frac{1}{r!}\log^r\frac{1}{1-z},\cr
\H_{\{1\}^r}(N)=&\Sum_{N\geq n_1>\ldots>n_r>0}\frac1{n_1\ldots n_r}& = \Sum_{k=0}^N \frac{ S_1(k, r) }{k!},\cr
\Li_{\{0\}^r}(z)=&\Sum_{n_1>\ldots>n_r>0}z^{n_1}&=\biggl(\Frac{z}{1-z}\biggr)^r,\cr
\H_{\{0\}^r}(N)=&\Sum_{N\geq n_1>\ldots>n_r>0}1&=\binom{N}{r},\cr
\Li_{-s_1,\ldots,-s_r}(z)=&\Sum_{n_1>\ldots>n_r>0}n_1^{s_1}\ldots n_r^{s_r}\;z^{n_1},&\cr
\H_{-s_1,\ldots,-s_r}(N)=&\Sum_{N\geq n_1>\ldots>n_r>0}n_1^{s_1}\ldots n_r^{s_r}.&
\end{eqnarray*}
Here, the Stirling numbers of first and second kind denoted $S_1(k,j)$ and $S_2(k,j)$ respectively, can be defined, for any $n, k\in\N,n\ge k$, by
\begin{eqnarray*}
\sum_{t=0}^n S_1(n,t) x^t=x(x+1)\ldots (x+n -1)&\mbox{and}&
S_2(n, k)=\frac{1}{k!}\sum_{i =0}^k (-1)^i\binom{k}{i}(k-i)^n.
\end{eqnarray*}

These divergent cases require the renormalization of the corresponding divergent polyzetas.
This is already done for the corresponding four first cases \cite{Daresbury,JSC,cade}
and it has to be completely done for the remainder \cite{FKMT,GZ,MP}.
Since the algebras of polylogarithms and of  harmonic sums, at strictly positive indices, are isomorphic respectively
to the shuffle, $(\QX,\shuffle,1_{X^*})$, and quasi-shuffle algebras, $(\Q\left\langle{ Y}\right\rangle,\stuffle,1_{Y^*})$,
both admitting, as pure transcendence bases, the Lyndon words $\Lyn X$ and $\Lyn Y$
over $X=\{x_0,x_1\}$ ($x_0<x_1$) and $Y=\{y_i\}_{i\ge1}$ ($y_1>y_2>\ldots$) respectively, we can index,
as in \cite{acta,VJM}, these polylogarithms, harmonic sums and polyzetas, at positive indices, by words.

Moreover, using
\begin{enumerate}
\item The one-to-one correspondence between the combinatorial compositions $(\{1\}^k,s_{k+1},\ldots,s_r)$,
the words $y_1^ky_{s_{k+1}}\ldots y_{s_r}$ and $x_1^kx_0^{s_{k+1}-1}x_1\ldots x_0^{s_r-1}x_1$ for
the indexing by words~:
$\Li_{x_0^{s_1-1}x_1\ldots x_0^{s_r-1}x_1}:=\Li_{s_1,\ldots,s_r}$
and $\H_{y_{s_1}\ldots y_{s_r}}:=\H_{s_1,\ldots,s_r}$.
Here, $\pi_Y$ is the adjoint of $\pi_X$ for the canonical scalar products where $\pi_X$ is the morphism of AAU,
$\ncp{k}{Y}\rightarrow \ncp{k}{X}$, defined by $\pi_X(y_k)=x_0^{k-1}x_1$.

\item The pure transcendence bases, denoted $\{S_l\}_{l\in\Lyn X}$ and $\{\Sigma_l\}_{l\in\Lyn Y}$,
of respectively $(\QX,\shuffle,1_{X^*})$ and $(\Q\left\langle{ Y}\right\rangle,\stuffle, 1_{Y^*})$;
and dually, the bases of Lie algebras of primitive elements $\{P_l\}_{l\in\Lyn X}$ and $\{\Pi_l\}_{l\in\Lyn Y}$
of respectively the bialgebras $(\QX,{\tt conc},1_{X^*},\Delta_{\shuffle},\epsilon)$
and $(\Q\left\langle{ Y}\right\rangle,{\tt conc},1_{Y^*},\Delta_{\stuffle},\epsilon)$ \cite{BDHHT,acta,VJM,reutenauer},

\item The noncommutative generating series in their factorized forms, {\it i.e.}
\begin{eqnarray*}
\L:=\prod_{l\in\Lyn X}\exp(\Li_{S_l}P_l)
&\mbox{and}&
\H:=\Prod_{l\in\Lyn Y}\exp(\H_{\Sigma_l}\Pi_l),\cr
Z_{\minishuffle}:=\prod_{l\in\Lyn X\setminus X}\exp(\zeta(S_l)P_l)
&\mbox{and}&
Z_{\ministuffle}:=\prod_{l\in\Lyn Y\setminus \{y_1\}}\exp(\zeta(\Sigma_l)\Pi_l),
\end{eqnarray*}
\end{enumerate}
we established an Abel like theorem \cite{cade,acta,VJM}, {\it i.e.} 
\begin{eqnarray}\label{Abel}
\Lim_{z\rightarrow1}\exp\biggl[y_1\log\frac1{1-z}\biggr]\pi_Y\L(z)=
\Lim_{N\rightarrow\infty}\exp\biggl[\Sum_{k\ge1}\H_{y_k}(N)\frac{(-y_1)^k}{k}\bigg]\H(N)=\pi_YZ_{\minishuffle}.
\end{eqnarray}
leading to discover a bridge equation for these two algebraic structures\footnote{
By definition, in $Z_{\minishuffle}$ and $Z_{\ministuffle}$, only {\it convergent}
polyzetas arise and and we do not need any regularization process, studied earlier in \cite{JSC, SLC43,SLC44}
and constructed in \cite{acta,VJM} (see also other processes developped in \cite{Boutet,IKZ}).}
\begin{eqnarray}\label{pont}
\exp\biggl[\sum_{k\ge2}\zeta(k)\frac{(-y_1)^k}{k}\biggr]
\Prod_{l\in\Lyn Y\setminus\{y_1\}}^{\searrow}\exp(\zeta(\Sigma_l)\Pi_l)
=\pi_Y\prod_{l\in\Lyn X\setminus X}^{\searrow}\exp(\zeta(S_l)P_l).
\end{eqnarray}
Extracting the coefficients in the generating series allows to explicit counter-terms which eliminate the divergence
of $\{\Li_w\}_{w\in x_1X^*}$ and $\{\H_w\}_{w\in y_1Y^*}$. Identifying local coordinates in (\ref{pont}), allows
to calculate the finite parts associated to divergent polyzetas and to describe the graded core of the kernel of $\zeta$
by {\it algebraic} generators\footnote{The graded core of a subspace $W$ is the largest graded subspace contained in $W$.
It is conjectured that this core is {\it all} the kernel. One of us proposed a tentative demonstration of this statement in \cite{acta,VJM}. 
If this holds, then this kernel would be generated by homogenous polynomials and the quotient be automatically $\N$-graded.}.

As in previous works, to study combinatorial aspects of harmonic sums and polylogarithms\footnote{
From now on, without contrary mention, it will be supposed that the indices of harmonic sums
and polylogarithms are taken as non-positive multi-indices.} at non-positive multi-indices, we associate
$(s_1,\ldots,s_r)\in{\mathbb N}^r$ to $y_{s_1}\ldots y_{s_r}\in Y^*_0$, where $Y_0 =\{y_k\}_{k\geq 0}$,
and index them by words (see Section \ref{combinatorics})~:
\begin{enumerate}
\item For $\H^-_{y_{s_1}\ldots y_{s_r}}:=\H_{-s_1,\ldots,-s_r}$, we will extend (Theorem \ref{degreeofHarmonic})
a form of Faulhaber's formula which expresses $\H^-_p$, for $p\ge0$, as a polynomial of degree $p+1$
with coefficients involving the Bernoulli numbers using the exponential generating series
$\sum_{n>0}(\sum_{k=1}^{n} k^{p}){z^n}/{n!}=(1-e^{pz})/(e^{-z}-1)$ \cite{Knuth}.
\item For $\Li^-_{y_{s_1}\ldots y_{s_r}}:=\Li_{-s_1,\ldots,-s_r}$, we will base ourselves (Theorem \ref{pr2})
on the Eulerian polynomials, $A_n(z)=\sum_{k=0}^{n-1}A_{n, k} z^{k}$ and the coefficients $A_{n,k}$'s
are the Eulerian numbers defined as $A_{n, k}=\sum_{j=0}^k(-1)^{j} \binom{n+1}{j}(k+1-j)^n$ \cite{GHMposter,foata}. 
\end{enumerate}

Afterwards, for their global renormalisation, we will consider their noncommutative generating series
and will establish an Abel like theorem (Theorem \ref{series})
analogous to \eqref{Abel}. Finally,  in Section \ref{structure}, to determine their algebraic structures
(Theorems \ref{onto0}, \ref{onto} and \ref{Tonghop}), we will construct a new law, denoted $\top$,
on $\Q\left\langle{Y_0}\right\rangle$, and will prove that the following morphisms of algebras,
mapping $w$ to $\H^-_w$ and $\Li^-_w$, respectively, are {\it surjective} and will completely describe their kernels
\begin{eqnarray}
\H^-_{\bullet}:(\ncp{\Q}{Y_0}, \stuffle)\longrightarrow(\Q\{\H^-_w\}_{w\in Y_0^*},.),\label{maps}\\
\Li^-_{\bullet}:(\ncp{\Q}{Y_0},\top)\longrightarrow(\Q\{\Li^-_w\}_{w\in Y_0^*},.).\label{maps2}
\end{eqnarray}

\section{Background}

\subsection{Combinatorial background of the quasi-shuffle Hopf algebras}\label{stuffle}
Let $Y_0$ be totally ordered by $y_0>y_1>\ldots$. We denote also $Y_0^+=Y_0Y_0^*$ and $Y^+=YY^*$,
the free semigroups of non-empty words.
The weight and length of $w=y_{s_1}\ldots y_{s_r}$,
are respectively the numbers $(w):=s_1+\ldots+s_r$ and $|w|:=r$.

Let $\K\left\langle{Y_0}\right\rangle$ be the vector space\footnote{or module, 
$\K$ is currently a ring.} freely generated by $Y_0^*$, {\it i.e.} $\K^{(Y_0^*)}$, equipped by
\begin{enumerate}
\item The concatenation (or by its associated coproduct, $\Delta_{\tt conc}$).
\item The {\it shuffle} product, {\it i.e.} the commutative product defined, for any $x,y\in Y_0$ and $u,v,w\in Y_0^*$, by 
\begin{eqnarray*}
w\shuffle 1_{Y_0^*}=1_{Y_0^*}\shuffle w=w&\mbox{and}&xu\shuffle yv=x(u\shuffle yv)+y(xu\shuffle v),
\end{eqnarray*}
or by its associated\footnote{in fact, this is the adjoint of the law.}  dual coproduct,
$\Delta_{\minishuffle}$, defined, on the letters $y_k\in Y_0$, by
$\Delta_{\minishuffle}(y_k)=y_k\otimes 1_{Y^*_0}+1_{Y_0^*}\otimes y_k$
and extended by morphism. It satisfies, for any $u,v,w\in Y_0^*$,
$\langle\Delta_{\minishuffle}(w)\mid u\otimes v\rangle=\langle w\mid u\shuffle v\rangle$.

\item The {\it quasi-shuffle} (or stuffle\footnote{This word was introduced by Borwein and {\it al.}.
The sticky shuffle product Hopf algebras were introduced independently
by physicists due to their r\^ole in stochastic analysis.}, or sticky shuffle) product,
{\it i.e.} the commutative product defined  by, for any $y_i,y_j\in Y_0$ and $u,v,w\in Y_0^*$, by 
\begin{eqnarray*}
w\stuffle 1_{Y_0^*}&=&1_{Y_0^*}\stuffle w =w,\\
y_iu\stuffle y_jv&=&y_j(y_iu\stuffle v)+y_i(u\stuffle y_jv)+y_{i+j}(u\stuffle v),
\end{eqnarray*}
or by its associated dual coproduct, $\Delta_{\ministuffle}$, defined on the letters $y_k\in Y_0$ by
\begin{eqnarray*}
\Delta_{\ministuffle}(y_k)=y_k\otimes 1_{Y^*_0}+1_{Y_0^*}\otimes y_k+\sum_{i+j=n}y_i\otimes y_j
\end{eqnarray*}
and extended by morphism. In fact, this is the adjoint of the law as it satisfies,
for all $u,v,w\in Y_0^*$, $ \scal{\Delta_{\ministuffle}(w)}{u\otimes v}=\scal{w}{u\stuffle v}.$
\item With the counit defined, for any $P\in\K \left\langle{Y_0}\right\rangle$, by
$\epsilon(P)=\scal{P}{1_{Y_0^*}}$, one gets\footnote{Out of these four bialgebras,
only the third one does not admit an antipode due to the presence of the group-like element $g=(1+y_0)$ which admits no inverse.}
\begin{eqnarray*}
\calH_{\minishuffle}=(\K\left\langle{Y_0}\right\rangle,{\tt conc}, 1_{Y_0^*},\Delta_{\shuffle},\epsilon)
&\mbox{and}&
\calH_{\minishuffle}^{\vee}=(\K \left\langle{Y_0}\right\rangle,\shuffle,1_{Y_0^*},\Delta_{\tt conc},\epsilon),\\
\calH_{\ministuffle}=(\K \left\langle{Y_0}\right\rangle,{\tt conc}, 1_{Y_0^*},\Delta_{\stuffle},\epsilon)
&\mbox{and}&
\calH_{\ministuffle}^{\vee}=(\K \left\langle{Y_0}\right\rangle,\stuffle,1_{Y_0^*},\Delta_{\tt conc},\epsilon).
\end{eqnarray*}
\end{enumerate}

\subsection{Integro-differential operators}\label{operators}
Let $\calC=\C[z,z^{-1},(1-z)^{-1}]$ and $\1_{\calC}:{\mathbb C}-(]-\infty,0]\cup[1,+\infty[)\rightarrow{\mathbb C}$ maps $z$ to $1$.
Let us consider the following differential and integration operators acting on $\calC\{\Li_w\}_{w\in X^*}$ \cite{acta}~:
\begin{eqnarray*}
\partial_z={d}/{dz},&\theta_0=z\partial_z,&\theta_1=(1-z)\partial_z,\\
\forall f\in\calC,\qquad\iota_0(f)=\int_{z_0}^z\Frac{f(s)ds}{s}&\mbox{and}&\iota_1 (f)=\int_{z_0}^z\Frac{f(s)ds}{1-s}.
\end{eqnarray*}
In here, $z_0=0$ if $\iota_0 f$ (resp. $\iota_1 f$) exists\footnote{This can be made precise, remarking that
$\calC\{\Li_w\}_{w\in X^*}=\calC\otimes_\C\C\{\Li_w\}_{w\in X^*}$ and then 
study the integral on the basis $\frac{z^k}{(1-z)^l}\otimes \Li_w$.} and else $z_0=1$.
One can check easily that 
$\theta_0+\theta_1=\partial_z$ and $\theta_0\iota_0=\theta_1\iota_1=\mathrm{Id}$ \cite{QED}.

For any $u=y_{t_1}\ldots y_{t_r}\in Y_0^*$, one can also rephrase the construction of polylogarithms as
$\Li_u=({\iota_0^{t_1-1}\iota_1\ldots\iota_0^{t_r-1}\iota_1})1_{\Omega}$, $\Li^-_u=({\theta_0^{t_1+1}\iota_1\ldots\theta_0^{t_r+1}\iota_1})1_{\Omega}$ and \cite{QED}
\begin{eqnarray}\label{noyaux}
\theta_0\Li_{x_0\pi_Xu}=\Li_{\pi_Xu};\theta_1\Li_{x_1\pi_Xu}=\Li_{\pi_Xu};\iota_0\Li_{\pi_Xu}=\Li_{x_0\pi_Xu};\iota_1\Li_u=\Li_{x_1\pi_Xu}.
\end{eqnarray}
The subspace $\calC\{\Li_w\}_{w\in X^*}$ (which is, in fact, a subalgebra) is then closed under the action
of  $\{\theta_0,\theta_1,\iota_0,\iota_1\}$ and the operators $\theta_0\iota_1$ and $\theta_1\iota_0$
admit respectively $\lambda$ (see \eqref{lambda}) and $1/\lambda$ as eigenvalues ($\calC\{ \Li_w\}_{w\in X^*}$ is their eigenspace) \cite{QED}~:
\begin{eqnarray}\label{remark}
\forall f \in\calC\{\Li_w\}_{w\in X^*},\quad(\theta_0\iota_1)f=\lambda f&\mbox{and}&(\theta_1\iota_0)f=f/\lambda,\\
\forall w\in X^*,\quad(\theta_0\iota_1) \Li_w=\lambda \Li_w &\mbox{and}&(\theta_1\iota_0)\Li_w=\Li_w/\lambda.
\end{eqnarray}
\subsection{Eulerian polynomials and Stirling numbers}
It is well-known that \cite{GHMposter,foata} for any $n \in \N_+$,
\begin{eqnarray}\label{Euler}
\Li^-_{y_n}(z)=\frac{z A_{n}(z)}{(1-z)^{n+1}}
&=&\sum_{k=0}^n\biggl(\sum_{j=k-1}^{n-1}A_{n,j}\binom{j+1}{k}(-1)^k\biggr)\frac{1}{(1-z)^{n+1-k}}\\
&=&\Sum_{t=1}^{n+1}\Frac{(t-1)!(-1)^{t+n+1}S_2(n+1,t)}{(1-z)^t}.
\end{eqnarray}
Using the notations given in Section \ref{operators}, one also has
\begin{eqnarray}\label{stirling}
\mbox{If }k>0&\mbox{then }\theta_0^k\lambda(z)=\Frac1{1-z}\Sum_{j=1}^kS_2(k,j)j!{\lambda^j(z)}&\mbox{else }\lambda(z).
\end{eqnarray}
Let $T$ be the invertible matrix $(t_{i,j})_{i\ge1}^{j\ge0}\in Mat_{\infty}(\Q_{\geq 0})$ defined by
\begin{eqnarray}\label{T}
\mbox{If }i > j&\mbox{then } t_{i,j}= \Frac{S_1(i,j+1)}{(i-1)!}&\mbox{else }0.
\end{eqnarray}

\section{Combinatorial aspects of harmonic sums and polylogarithms}\label{combinatorics}
\subsection{Combinatorial aspects of harmonic sums at non-positive  multi-indices}\label{Hcombinatorics}
 
\begin{definition}[Extended Bernoulli polynomials]
Let  $\{B_w\}_{w\in Y^*_0},\{\beta_w\}_{w\in Y^*}$ be two families of polynomials defined,
for any $r\ge1, z\in\mathbb{C}, y_{n_1}\ldots y_{n_r}\in Y_0^*$ by
\begin{eqnarray*}
B_{y_{s_1}\ldots y_{s_r}}(z+1)=B_{y_{s_1}\ldots y_{s_r}}(z)+s_1 z^{s_1 -1}B_{y_{s_2}\ldots y_{s_r}}(z)
\end{eqnarray*}
and, for any 
$y_{n_1}\ldots y_{n_r}\in Y^*$ we set $\beta_w(z):=B_w(z)-B_w(0)$.

One defines also, for any $w=y_{n_1}\ldots y_{n_r}\in Y_0^*$ (with $b_w:=B_w(0)$) and
\begin{eqnarray*}
b'_{y_{n_k}}:=b_{y_{n_k}}&\mbox{and}&
b'_{y_{n_k}\ldots y_{n_r}}:=b_{y_{n_k} \ldots y_{n_r}}-\sum_{j=0}^{r-1-k}b_{n_{y_{k+j+1}}\ldots y_{n_{r}}}b'_{y_{n_k}\ldots y_{n_{k+j}}}.
\end{eqnarray*}
\end{definition}
Note that, for any $s_1\neq1$, one has $B_{y_{s_1}\ldots y_{s_r}}(0)=B_{y_{s_1}\ldots y_{s_r}}(1)=b_{y_{s_1}\ldots y_{s_r}}$
and the polynomials $\{B_w\}_{w\in Y^*_0}$ depend on the (arbitrary) choice of $\{b_w\}_{w\in Y_0^*}$. In here, by \eqref{pair}, we put $b_{y_{s_1}}=b_{s_1}$ then $B_{y_{s_1}}(z)=B_{s_1}(z)$ is the $s_1^{th}-$ Bernoulli polynomial \cite{words03}.

Now, let $M = \begin{pmatrix}m_{i,j}\end{pmatrix}^{i \geq 0}_{j\geq 1}\in Mat_{\infty} (\Q)$ be the invertible matrix\footnote{In this paper, $M^{-1}$ and $M^t$ denote as usual the inverse and transpose of the matrix $M$.}
 defined by 
\begin{eqnarray}\label{M}
m_{i,j} =
\begin{cases}
0,& \mbox{if} \quad i < j-1,\\
b_i,& \mbox{if} \quad j= 1, i\neq 1\\
1/2,& \mbox{if} \quad j= i = 1\\
im_{i-1,j-1}/j,& \mbox{if} \quad i > j-1 > 0.
\end{cases}
\end{eqnarray}
Let us consider also the following invertible matrix $D\in Mat_r(\N)$
$$d_{i,j}=\left\{\begin{array}{ccl}
0,&\mbox{if}&1\le i < j \leq r,\cr
n_1\ldots n_i,&\mbox{if}&1\le i=j\le r,\cr
n_1\ldots n_jb_{y_{n_{j+1}}\ldots y_{n_{i}}},&\mbox{if}&1\le j< i\le r,
\end{array}\right.$$
Then its inverse,  $D^{-1}=(v_{i,j})$, can be also described as follows
$$v_{i,j}=\left\{\begin{array}{ccl}
0,&\mbox{if}&1\le i<j\le r,\\
1/(n_1 \ldots n_i),&\mbox{if}&1\le i=j\le r,\\
-b'_{y_{n_{j+1}}\ldots y_{n_i}}/(n_1\ldots n_i),&\mbox{if}&1\le  j<i\le r.
\end{array}\right.$$

\begin{proposition}[\cite{GHMposter}]
For any $N>0$ and $y_{n_1}\ldots y_{n_r}\in Y^*$, one has
\begin{eqnarray*}
\beta_{y_{n_1}\ldots y_{n_r}}(N+1)=
\sum_{k=1}^r(\prod_{i=1}^kn_i)b_{y_{n_{k+1}}\ldots y_{n_r}}\H^-_{y_{n_1-1}\ldots y_{n_k-1}}(N).
\end{eqnarray*}
\end{proposition}
\begin{proof}
Successively $B_{y_{n_1}\ldots y_{n_r}}(N+1)-B_{y_{n_1}\ldots y_{n_r}}(N)=n_1N^{n_1-1}B_{y_{n_2}\ldots y_{n_r}}(N),\allowbreak
\ldots, B_{y_{n_1}\ldots y_{n_r}}(2)-b_{y_{n_1}\ldots y_{n_r}}=n_11^{n_1-1}B_{y_{n_2}\ldots y_{n_r}}(1)$.
It follows the expected result:
\begin{eqnarray*}
\beta_{y_{n_1} \ldots y_{n_r}}(N+1)
&=&n_1\Sum_{k_1 =1}^Nk_1^{n_1-1}B_{y_{n_2}\ldots y_{n_r}}(k_1)\cr
&=&n_1\Sum_{k_1 =1}^Nk_1^{n_1-1}\beta_{y_{n_2} \ldots y_{n_r}} (k_1)+n_1 \sum_{k_1 =1}^Nk_1^{n_1 - 1}b_{y_{n_2}\ldots y_{n_r}}\cr
&\vdots&\cr
&=&\sum_{k=1}^r(\prod_{i=1}^kn_i)b_{y_{n_{k+1}}\ldots y_{n_r}}\H^-_{y_{n_1-1}\ldots y_{n_k-1}}(N).
\end{eqnarray*}
\end{proof}

\begin{theorem}\label{degreeofHarmonic}
\begin{enumerate}
\item For any $w\in Y^*_0, \H^-_w$ is a polynomial function , in $N$, of degree $(w)+|w|$.
\item If $w\in Y^*$, there exists a polynomial $\G^-_w$, of degree $(w)-1$,
such that $\H^-_w(N)=(N+1)N(N-1) \ldots (N-|w|+1)\G^-_w(N).$
Conversely, for any $N,k\in\N_+$, one has $N^k=\sum_{j =0}^{k-1}(-1)^{j+k-1}\binom{k}{j}\H^-_{y_j}(N)$.
\item We get (for $r=1$, this corresponds to Faulhaber's formula \cite{Faulhaber1,Knuth})~:
\begin{eqnarray*}
\H^-_{y_{n_1}\ldots y_{n_r}}(N)=\frac{\beta_{y_{n_1+1}\ldots y_{n_r +1}}(N+1)
-\sum_{k=1}^{r-1}b'_{y_{n_{k+1}+1}\ldots y_{n_r +1}}\beta_{y_{n_1+1}\ldots y_{n_k +1}}(N+1)}{\prod_{i=1}^r (n_i +1)}.
\end{eqnarray*}
\end{enumerate}
\end{theorem}
\begin{proof}
For any $y_n\in Y$ and $w=y_ny_m\in Y^*$, putting $p =(w)+|w|$, we have
\begin{eqnarray*}
\H^-_{y_{n}}(N)&=&\Frac1{n+1}\Sum_{k=0}^{n}\binom{n +1}{k}b_{k}(N+1)^{n+1-k},\\
\H^-_{y_ny_m}(N)&=&\sum_{k=0}^m\sum_{l=0}^{p-1-k}\sum_{q=0}^{p-k-l}\frac{b_kb_l}{(m+1)(p-k)}
\binom{m+1}{k}\binom{p-k}{l}\binom{p-k-l}{q}N^q.
\end{eqnarray*}
\begin{enumerate}
\item As above, $\H^-_{y_n},\H^-_{y_ny_m}$ are polynomials of respective degrees $n+1,n+m+2$.

Since, for any $s\in\N,w\in Y^*_0,N\in\N_+$, $\H^-_{y_sw}(N+1)-\H^-_{y_sw}(N)=N^s\H^-_w(N)$ then
$\H^-_{y_sw}$ satisfies the difference equation  $f(N+1)-f(N)=N^s\H^-_w(N)$ (see Lemma \ref{solutions_diff} below).
For $|w|=k>2$, suppose $\H^-_w$ is a polynomial and $\deg(\H^-_w)=(w)+|w|$.
Now, let $y_s \in Y_0$. Then $\H^-_{y_sw}$ is not constant. Indeed,
$\H^-_{y_sw}(N)=\sum_{n=1}^N n^s\H^-_w (n-1)$.
Moreover, $\H^-_{y_sw}(N+1)-\H^-_{y_sw}(N)=(N+1)^s\H^-_w(N)$.
Thus, by Lemma \ref{solutions_diff}, $\deg(\H^-_{y_sw})=\deg(\H^-_{w})+s+1=(y_sw)+|y_sw|$.

By the definition, for any $w=y_{s_1}\ldots y_{s_r} \in Y^*,\{0,\ldots,r-1\}$ are solutions of the equation $\H^-_w\equiv0$.
There exists then a polynomial $G_w$ such that $\H^-_{y_{s_1}\ldots y_{s_r}}(N)=N(N-1)\ldots(N-r+1)G_w(N)$.
Now, for any $s\ge1$ and $w\in Y^*$, we have $\H^-_{y_sw} (N+1)-\H^-_{y_sw}(N)=(N+1)^s\H^-_w(N)$.
So $-1$ is also solution of the equation $\H^-_w\equiv0$. Then the second result follows.
\item By Faulhaber's formula \cite{Faulhaber1,Knuth}, one has
$\begin{pmatrix}\H^-_{y_i}(N)\end{pmatrix}^t_{i\ge0}=M \begin{pmatrix}N^j\end{pmatrix}^t_{j\ge1}$.
Thus, by inversion matrix, the third result follows.
\item Let $H^-:=\begin{pmatrix}\H^-_{y_{n_1-1}}&\ldots&\H^-_{y_{n_1-1}\ldots y_{n_r -1}}\end{pmatrix}^t$ and
$\beta:=\begin{pmatrix}\beta_{y_{n_1}}&\ldots&\beta_{y_{n_1} \ldots y_{n_r}}\end{pmatrix}^t$.
Hence, $H^-(N)=D^{-1}\beta(N+1)$ and the last result follows.
\end{enumerate}
\end{proof}

\begin{example}\label{exampleharmonic}
\begin{itemize}
\item For $r=1$, 
\begin{eqnarray*}
\H^-_{y_0}(N)\!\!\!&=&\!\!\!N,\cr
\H^-_{y_1}(N)\!\!\!&=&\!\!\!{N(N+1)}/{2},\cr
\H^-_{y_2}(N)\!\!\!&=&\!\!\!{N(N+1)(2N+1)}/{6},\cr
\H^-_{y_3}(N)\!\!\!&=&\!\!\!{N^2(N+1)^2}/{4}.
\end{eqnarray*}
\item For $r=2$,
\begin{eqnarray*}
\H^-_{y_0^2}(N)\!\!\!&=&\!\!\!{N(N-1)}/{2},\cr
\H^-_{y_1^2}(N)\!\!\!&=&\!\!\!{N(N-1)(3N+2)(N+1)}/{24},\cr
\H^-_{y_1y_2}(N)\!\!\!&=&\!\!\!{N(N-1)(N+1)(8N^2+5N-2)}/{120},\cr
\H^-_{y_2y_1}(N)\!\!\!&=&\!\!\!{N(N-1)(N+1)(12N^2+15N+2)}/{120}.
\end{eqnarray*}
\item For $r=3$,
\begin{eqnarray*}
\H^-_{y_0^3}(N)\!\!\!&=&\!\!\!{N(N-1)(N-2)}/{6},\cr
\H^-_{y_1^3}(N)\!\!\!&=&\!\!\!{N^2(N-1)(N-2)(N+1)^2}/{48},\cr
\H^-_{y_1^2y_2}(N)\!\!\!&=&\!\!\!{N(N-1)(N-2)(N+1)(48N^3+19N^2-61N-24)}/{5040},\cr
\H^-_{y_1^2y_3}(N)\!\!\!&=&\!\!\!{N(N-1)(N-2)(N+1)(7N^2+3N-2)(5N^2-3N-12)}/{6720}.
\end{eqnarray*}
\end{itemize}
\end{example}

\subsection{Combinatorial aspects of polylogarithms at non-positive multi-indices}\label{Lcombinatorics}
\begin{definition}[Extended Eulerian polynomials]\label{Eulerianpolynomial} 
For any $w=y_{s_1}\ldots y_{s_r}\in Y^+$, the polynomial $A^-_w$, 
of degree $(w)$, is defined as follows
\begin{eqnarray*}
\mbox{If }r=1,&A^-_w=A_{s_1}\mbox{else } A^-_w=\Sum_{i=0}^{s_1}
\binom{s_1}{i} A^-_{y_i}A^-_{y_{s_1+s_2-i}y_{s_3}\ldots y_{s_r}}&
\end{eqnarray*}
where, for any $n\in\N$, $A_n$ denotes the $n$-th classical Eulerian polynomial.
\end{definition}
Note that the coefficients of the polyomials $\{A_w^-\}_{w\in Y^*_0}$ are integers.
\begin{theorem}\label{pr2}
\begin{enumerate}
\item If $r>1$ then $\Li^-_{y_{s_1}\ldots y_{s_r}}=\sum^{s_1}_{t=0}\binom{s_1}{t}\Li^-_{y_t}\Li^-_{y_{s_1+s_2-t}y_{s_3}\ldots y_{s_r}}.$
\item For any $w\in Y_0^*$, $\Li^-_w(z)=\lambda^{\abs{w}} (z)A^-_w(z)(1-z)^{-(w)}\in\Z[(1-z)^{-1}]$ is a polynomial function of degree\footnote{
Putting $K_w({1}/{z}):={A^-_w(z)}/{z^{(w)}}$, one gets $\Li^w_w (z)=\lambda^{(w)+|w|} (z)K_w({1}/{z})$.} $(w)+\abs{w}$ on $(1-z)^{-1}$.
Conversely, we get
$(1-z)^{-k}=(1-z)^{-1}+\sum_{j=2}^k S_1(k,j)\Li^-_{y_{j-1}}(z)/(k-1)!$.
\item We have
$\Li^-_{y_{s_1}\ldots y_{s_r}}(z)=\lambda(z)^{\abs w}\sum_{i=r}^{s_1+\ldots+s_r}\sum_{j=0}^{s_1\ldots+s_{r-1}}l_{i,j}{z^{i-1-j}}{(1-z)^{-i}}$, where
\begin{eqnarray*}
l_{i,j}&=&\sum\limits_{\underset{k_1+\ldots + k_r = i}{1\leq k_t \leq s_t}} \prod\limits_{n =1}^r(k_n!S_2(s_n,k_n))
\sum\limits^{t_1+\ldots+t_{r-1} = j}_{\underset{1\leq m \leq r-1}{0 \leq t_m\leq k_m}}\prod\limits_{p=1}^{r-1}\\
&&\binom{k_r+\ldots+k_{r-p+1}+p-t_{r-p+1}-\ldots-t_{r-1}}{t_{r-p}}\binom{k_{r-p}+t_{r-p+1}+\ldots t_{r-1}}{k_{r-p}-t_{r-p}}.
\end{eqnarray*}
\end{enumerate} 
\end{theorem}

\begin{proof}
For $w=y_{s_1}\ldots y_{s_r}\in Y^*_0$, $\Li^-_w=(\theta_0^{s_1+1}\iota_1)\ldots(\theta_0^{s_{r-1}+1}\iota_1)\Li^-_{y_{s_r}}$ \cite{QED}. Hence, 
\begin{enumerate}
\item The actions of $\theta_i,\iota_i$ yield immediately the expected result.
Next, using $\theta_i,\iota_i$, one has 
$\Li^-_w(z)=\lambda^{\abs{w}}(z)A^-_w(z)(1-z)^{-(w)}$ and $\Li^-_w\in\Z[(1-z)^{-1}]$.
Hence,  $\Li^-_w$ is a polynomial of degree $(w)+|w|$ $(w)+\abs{w}$ on $(1-z)^{-1}$. 
\item Since $ \begin{pmatrix}(1-z)^{-j}\end{pmatrix}^t_{j\ge1}= T \begin{pmatrix}(1-z)^{-1}&(\Li^-_{y_i}(z))_{i\ge1}\end{pmatrix}^t$
then, by matrix inversion, the expected result follows immediately.
\item The expected result follows by using \eqref{stirling} in the following expressions
\begin{eqnarray*}
\Li^-_{y_{s_1}\ldots y_{s_r}}
&=&\sum_{k_1=0}^{s_1}\sum_{k_2=0}^{s_1+s_2-k_1}\ldots\sum_{k_r=0}^{(s_1+\ldots+s_r)-\atop(k_1+\ldots+k_{r-1})}
\binom{s_1}{k_1}\binom{s_1+s_2-k_1}{k_2}\ldots\\
&&\binom{s_1+\ldots+s_r-k_1-\ldots-k_{r-1}}{k_r}(\theta_0^{k_r}\lambda)(\theta_0^{k_2}\lambda)\ldots(\theta_0^{k_r}\lambda).
\end{eqnarray*}
\end{enumerate} 
\end{proof}
\goodbreak

\begin{example}
Since $\Li^-_{y_my_n}=(\theta_0^{m+1}\iota_1)\Li^-_{y_n}=\theta_0^m(\theta_0\iota_1)\Li^-_{y_n}$
and $(\theta_0\iota_1)\Li^-_{y_n}=\Li_0\Li^-_{y_n}$, one has $\Li^-_{y_my_n}=\theta_0^m[\Li_0\Li^-_{y_n}]
=\sum_{l=0}^m{m\choose l}\Li^-_{y_l}\Li^-_{y_{m+n-l}}$. For example, 
\begin{eqnarray*}
\Li^-_{y_1^2} (z)
&=&\Li_{y_0}(z)\Li^-_{y_2}(z)+(\Li^-_{y_1}(z))^2\\
&=&-(1-z)^{-1}+{5}(1-z)^{-2}-{7}(1-z)^{-3}+ {3}(1-z)^{-4},\cr
\Li^-_{y_2y_1}(z)&=&\Li_{y_0}(z)\Li^-_{y_3}(z) +3\Li^-_{y_1}(z)\Li^-_{y_2}(z)\cr
&=&(1-z)^{-1}-{11}(1-z)^{-2}+{31}(1-z)^{-3}-{33}(1-z)^{-4}+{12}(1-z)^{-5},\cr
\Li^-_{y_1y_2}(z)&=&\Li_{y_0}(z)\Li^-_{y_3}(z)+\Li^-_{y_1}(z)\Li^-_{y_2}(z)\cr
&=&(1-z)^{-1}-{9}(1-z)^{-2}+{23}(1-z)^{-3}-{23}(1-z)^{-4}+{8}(1-z)^{-5}.
\end{eqnarray*}
\end{example}

\subsection{Asymptotics of harmonic sums and singular expansion of polylogarithms}\label{asymptotic}

\begin{proposition}[\cite{GHMposter}]\label{P3}
Let $w\in Y^*_0$. There are non-zero constants $C^-_w$ and $B^-_w$ such that\footnote{
This means $\lim_{N\rightarrow+\infty}N^{-((w)+|w|)}\H^-_w(N)=C^-_w$ and $\lim_{z\rightarrow+1}(1-z)^{(w)+|w|}\Li^-_w(z)=B^-_w$.} 
$\H^-_w(N){}_{\widetilde{N\rightarrow\infty}}N^{\left(w\right)+|w|}C^-_{w}$ and $\Li^-_w(z){}_{\widetilde{z\rightarrow1}}B^-_w(1-z)^{-((w)+|w|)}$.

Moreover, $C^-_{1_{Y^*_0}}=B^-_{1_{Y^*_0}}=1$ and, for any $w\in Y_0^+$,
\begin{eqnarray*}
C^-_w=\prod_{w=uv, v\neq1_{Y^*}} ((v)+|v|)^{-1}\in\Q&\mbox{and}&B^-_w=((w)+\abs{w})!C^-_w\in\N_+.
\end{eqnarray*}
\end{proposition}

\begin{proof}
By Theorem \ref{degreeofHarmonic}, $\H^-_w$ is a polynomial function of degree $(w)+\abs{w}$ then such
$C^-_w$ exists. By induction, it is immediate that $C^-_{1_{Y^*_0}}=1$ and $C^-_{y_s}=(s+1)^{-1}$,
{\it i.e.} the result holds for $|w|\le1$. Suppose that it holds up to $|w|=k$ and for the next,
let $y_s\in Y_0$ and $w\in Y^*_0$ such that $|w|=k$. Since $\H^-_{y_sw}$ is solution of $f(N+1)-f(N)=(N+1)^s\H^-_w(N)$,
with initial condition $\H^-_w(0)=0$, then we get the leading term of $\H^-_w$.
Remark that $(N+1)^s\H^-_{w}(N)\not\equiv0$ is polynomial with $C^-_wN^{s+(w)+|w|}$ as leading term. 
Hence, by Lemma \ref{solutions_diff}, the leading term of $\H^-_{y_sw}$ is $C^-_wN^{(y_sw)+|y_sw|}/((y_sw)+|y_sw|)$.
It follows from this the expression of $C^-_w$. The second result is immediate by Theorem \ref{pr2}.
\end{proof}

\begin{example}[of $C^-_w$ and $B^-_w$]
$$\begin{tabular}{|c|c|c|c|c|c|}
\hline
$w$& $C^-_w$ & $B^-_w$ & $w$& $C^-_w$ & $B^-_w$\\ 
\hline
$y_0$ &$1$ & $1$ & $y_1y_2$ & $\frac{1}{15}$ & $8$\\ 
 \hline
$y_1$ & $\frac{1}{2}$ & $1$ & $y_2y_3$ & $\frac{1}{28}$ & $180$\\ 
 \hline
$y_2$ &$\frac{1}{3}$ & $2$ & $y_3y_4$ & $\frac{1}{49}$ & $8064$\\ 
 \hline
$y_n$ & $\frac{1}{n+1}$ & $n!$ & $y_my_n$ & $\frac{1}{(n+1)(m+n+2)}$ & $\frac{(m+n+1)!}{(n+1)}$\\ 
 \hline
$y_0^2$ &$\frac{1}{2}$ & $1$ & $y_2y_2y_3$ & $\frac{1}{280}$ & $12960$\\ 
 \hline
$y_0^n$ &$\frac{1}{n!}$ & $1$ & $y_2y_{10}y_1^2$ & $\frac{1}{2160}$ & $9686476800$\\ 
 \hline
$y_1^2$ & $\frac{1}{8}$ & $3$ & $y_2^2y_4y_3y_{11}$ & $\frac{1}{2612736}$ & $4167611825465088000000$\\ 
 \hline
\end{tabular}$$
\end{example}

\begin{definition}\label{C}
Let us define the two following non commutative generating series \footnote{using the Kleene star of series without constant term $S^* = (1-S)^{-1}$.}
\begin{eqnarray*}
C^-:=\sum_{w\in Y_0^*}C^-_w\;w&\mbox{and}&\Theta(t):=\sum_{w\in Y_0^*}t^{|w|+(w)}w=\biggl(\Sum_{y\in Y_0}t^{(y)+1}y\biggr)^*,\\
B^-:=\sum_{w\in Y_0^*}B^-_w\;w&\mbox{and}&\Lambda(t):=\sum_{w\in Y_0^*}{(|w|+(w))!}{ t^{|w|+(w)}}w.
\end{eqnarray*}
\end{definition}
By Theorems \ref{degreeofHarmonic}, \ref{pr2} and, for any $w\in Y^*_0$, denoting $p=(w)+|w|$, one also has
\begin{eqnarray*}
\Theta(N)&=&1_{Y^*_0}+ \Sum_{w \in Y^+_0}\biggl[\Sum_{j=0}^{p-1}(-1)^{p+j-1}\binom{p}{j}\H^-_{y_j}(N)\biggr]w,\cr
\Lambda((1-z)^{-1})&=&1_{Y^*_0}+ (\Li^-_{y_0}(z)-1)y_0\\ 
&+&\Sum_{w\in Y^+_0 \setminus \{y_0\}}\biggl[(-1)^{p+1}(\Li^-_{y_0}(z)-1)+\Sum_{j=2}^p\Frac{(-1)^{p+j}S_1(p,j)}{(p-1)!}\Li^-_{y_{j-1}}(z)\biggr]w.
\end{eqnarray*}

By Definition \ref{C} and Proposition \ref{P3}, we get
\begin{theorem}\label{series}
Let us consider now the following noncommutative generating series
\begin{eqnarray*}
\H^-:=\sum_{w\in Y_0^*}\H^-_w\;w&\mbox{and}&\L^-:=\sum_{w\in Y_0^*}\Li^-_w\;w.
\end{eqnarray*}
Then\footnote{{\it i.e.},
$\H^-(N)\;{}_{\widetilde{N\rightarrow+\infty}}\;C^-\odot\Theta(N)$ and $\L^-(z)\;{}_{\widetilde{z\rightarrow1}}\;C^-\odot\Lambda((1-z)^{-1})$
where $\Theta^{\odot-1}$ and $\Lambda^{\odot-1}$ denote the inverses for Hadamard product of $\Theta$ and $\Lambda$ respectively
(to be compared with \eqref{Abel}). The equivalences are understood term by term.}
$\lim_{N \to+\infty}\Theta^{\odot -1}(N)\odot\H^-(N)=\lim_{z \to 1}\Lambda^{\odot -1}((1-z)^{-1})\odot\L^-(z)=C^-$.
\end{theorem}

\begin{definition}
Let $P \in\Q\left\langle{Y_0}\right\rangle$ such that $\H^-_P\not\equiv 0$. Let $B^-_P, C^-_P\in\Q$ and $n(P)\in\N$
be defined respectively by $\Li^-_P(z){}_{\widetilde{z\rightarrow1}}(1-z)^{-n(P)} B^-_P$ and $\H^-_P(N){}_{\widetilde{N\rightarrow\infty}}N^{n(P)}C^-_P$.
\end{definition}

It is immediate that\footnote{The support of $P=\sum_{u \in Y^*_0}x_uu\in\Q\left\langle{Y_0}\right\rangle$ is defined by
$\mathrm{supp}P=\{w\in Y_0^*|\scal{P}{w}\neq 0\}$ and the scalar product, on $\Q\left\langle{Y_0}\right\rangle$,
is defined by $\scal{P}{w}=x_w$.} $0\le n(P)\le\max_{u\in\mathrm{supp}P}\{(u)+\abs{u}\}$.

Let $p_{\max}:=\max_{u\in\mathrm{supp}P}\{(u)+|u|\}<+\infty$.
We are extending\footnote{This process, over $\Q\left\langle{Y_0}\right\rangle$, ends after a finite number of steps.}
$C^-_{\bullet}$ over $\Q\left\langle{Y_0}\right\rangle$~:
\begin{itemize}
\item Step 1: Compute $C_0=\sum_{(w)+|w|=p_{\max}}\scal{P}{w}C^-_w$. If $C_0\neq 0$ then $C^-_P=C_0$ else go to next step.
\begin{example}
For $P=6y_4y_2+12y_3^2-9y_5$, $(y_4y_2)+\abs{y_4y_2}=(y_3^2)+\abs{y_3^2}=8>6=(y_5)+|y_5|$.
Then $p_{\max}=8,C_0=6C^-_{y_4y_2}+12C^-_{y_3^2}=\frac{5}{8}$ and $C^-_P=C_0={5}/{8}$.
\end{example}

\item Step 2: Let $p:=p_{\max}-1$ and $J_1:=\{v\in\mathrm{supp}P|(v)+\abs{v}=p\}$. Compute $C_1=\sum_{c\in J_1}\alpha_cC^-_c
\scal{\H^-_{\sum_{c \in J}\alpha_cc}}{N^p}$. If $C_1\neq0$ then $C_0:=C_1$ else go on with $p-1$.
\begin{example}
$P=12y_2y_1^4 - y_2y_{3}^2 - 9y_4^2$,
$(y_1^4y_2)+\abs{y_1^4y_2}=(y_2y_3^2)+\abs{y_2y_3^2}=11>10=(y_4^2)+\abs{y_4^2}$.
Thus, $p_{\max}=11,C_0=12C^-_{y_2y_1^4}-C^-_{y_2y_3^2}=0$. Next step. One has
\begin{eqnarray*}
\H^-_{y_2y_1^4}(N)&=&\frac{1}{4224}N^{11}-\frac{181}{134400}N^{10}+\frac{31}{80640}N^9+\frac{43}{5376}N^8-\frac{1}{168}N^7\\
&-&\frac{1087}{57600}N^6+\frac{47}{3840}N^5+\frac{323}{16128}N^4-\frac{1}{126}N^3-\frac{787}{100800}N^2+\frac{19}{18480}N,\cr
\H^-_{y_2y_3^2}(N)&=&\frac{1}{352}N^{11}-\frac{9}{2240}N^{10}-\frac{95}{4032}N^9+\frac{11}{448}N^8+\frac{13}{168}N^7-\frac{149}{2880}N^6\cr
&-&\frac{37}{320}N^5+\frac{173}{4032}N^4+\frac{71}{1008}N^3-\frac{59}{5040}N^2-\frac{53}{4620}N.
\end{eqnarray*}
Hence, $C_1=-\Frac{2172}{134400}+\Frac{9}{2240}-9C^-_{y_4^2}=-\Frac{269}{1400}\ne0$ and then $C^-_P=-\Frac{269}{1400}$.
\end{example}

\item Suppose, in the $r^{th}-$ step, that the linear combination of
coefficients of $N^{p_{\max}-r+1}$ in $\H^-_w$
with $w\in\{u\in\mathrm{supp}P|(u)+\abs{u}\ge p_{\max}-r+1\}$
 is zero, namely $C_{r-1}=0$,
then in $(r+1)^{th}-$ step to compute the linear combination of coefficients of $N^{p_{\max}-r}$ in $\H^-_w$
with $w\in\{u\in\mathrm{supp}P|(u)+|u|\ge p_{\max}-r+1\}$,
and the linear combination of $C^-_w$ with $w\in\{u \in\mathrm{supp}P|(u)+\abs{u}=p_{\max}-r\}$.
If the sum of above results is not $0$ then it is $C^-_P$ else one jumps to next step.  
\end{itemize}

For $P=\sum_{w\in Y^*}c_ww$, the similar algorithm can be used to compute $B^-_P$~:
\begin{itemize}
\item Step 1: Let $J:=\{v\in\mathrm{supp}P|(v)+\abs{v}=p_{\max}\}$.
Compute $B^0=\sum_{c\in J }\alpha_cB^-_c$. If $B^0 \neq 0$, then $B^-_P=B^0$  else go to next step.
\item Step 2: Let $J_1=\{v\in\mathrm{supp}P|(v)+\abs{v}= p_{\max}-1\}$.
Find the coefficient of $(1-z)^{-p_{\max}+1}$ in $\Li^-_{\sum_{c\in J}\alpha_cc}$, namely $b_{p_{\max}-1}$,
and compute $B^1=b_{p_{\max}-1}+\sum_{c\in J_1} \alpha_cB^-_c$.
If $B^1\neq0$ then $B^-_P=B^1$ else continue with smaller orders.
\end{itemize}

\section{Structure of harmonic sums and polylogarithms}\label{structure}
\subsection{Structure of harmonic sums at non-positive multi-indices}
\begin{proposition}
Let $w_1,w_2\in Y^*_0$. Then $\H^-_{w_1}\H^-_{w_2}=\H^-_{w_1\stuffle w_2}.$
\end{proposition}
\begin{proof}
Associating $(s_1,\ldots,s_k)$ to $w$, the quasi-symmetric monomial function on the commuting variables $t=\{t_i \}_{i\ge1}$ is defined by
\begin{eqnarray*}
M_{1_{Y^*}}(t)=1&\mbox{and}&M_w(t)=\sum_{n_1>\ldots>n_k> 0} t^{s_1}_{n_1}\ldots t^{s_k}_{n_k}.
\end{eqnarray*}
For any $u,v \in Y^*_0$, since $M_uM_v=M_{u\stuffle v}$ then $\H^-_{s_1,\ldots,s_k}$ is obtained by specializing, in $M_w$,
$t=\{t_i \}_{i\ge1}$ at $t_i=i$ if $1\le i \le N$ and else at 
$t_i = 0$ for $i \geq N+1$ for $i \geq N+1$.
\end{proof}

By the $\stuffle$-extended Friedrichs criterion \cite{acta,VJM}, we get
\begin{theorem}\label{onto0}
\begin{enumerate}
\item The generating series $\H^-$ is group-like and $\log\H^-$ is primitive.
\item $\ker\H^-_{\bullet}$ is a prime ideal of $(\Q \left\langle{Y_0}\right\rangle,\stuffle)$,
{\it i.e.} $\ncp{\Q}{Y_0}\setminus\ker\H^-_{\bullet}$ is closed by $\stuffle$.
\end{enumerate}
\end{theorem}

\begin{definition}\label{blunt cones}
For any $n\in\N_+$, let
$\mathbb{P}_n:=\mathrm{span}_{\R_+}\{w\in Y_0^*|(w)+|w|=n\}\setminus\{0\} \subset \ncp{\R_+}{Y_0}$
be the blunt convex cone ({\it i.e.} without zero or see Appendix A.) generated by the set $\{w\in Y_0^*|(w)+|w|=n\}$.
\end{definition}

By definition, $C^-_{\bullet}$ is linear on the set $\mathbb{P}_n$. For any $u, v \in Y^*_0$, one has $u\stuffle v=u\shuffle v+\sum_{(w)+|w|<(u)+|u|+|v|+(v)}x_ww$ and the $x_w$'s are positive. Moreover, for any $w$ which belongs to the support of 
$\sum_{(w)+|w|<(u)+|u|+|v|+(v)}x_ww$,  
one has $(w)+|w| < (u)+(v)+|u|+|v|$. Thus, by the definition of $C^-_{\bullet}$, one obtains
\begin{corollary}\label{tac}
\begin{enumerate}
\item Let $w,v \in Y^*_0$. Then $C^-_wC^-_v=C^-_{w\shuffle v }=C^-_{w\stuffle v }$.
\item For any $P, Q \notin\ker\H^-_{\bullet}$, $C^-_PC^-_Q=  C^-_{P \stuffle Q }$ and
$\Q\left\langle{Y_0}\right\rangle\setminus\ker\H^-_{\bullet}$ is a $\stuffle-$ multiplicative monoid containing $Y^*_0$.
\end{enumerate}
\end{corollary}
Let us prove $C^-_{\bullet}$ can be extended as a character  $C^-$, for $\shuffle$ or equivalently, by the Freidrichs' criterion \cite{reutenauer},
$C^-$ is group-like and then $\log{C}^-$ is primitive, for $\Delta_{\shuffle}$.
\begin{lemma}\label{HGM}
Let $\mathcal{A}$ is a $\R$-associative algebra with unit and
$f:\sqcup_{n\ge0}\mathbb{P}_{n}\longrightarrow\mathcal{A}$ (see Definition \ref{blunt cones}) such that 
\begin{enumerate}
\item For any $u, v \in Y^*_0$, $f(u \shuffle v)=f(u)f(v)$ and
 $f(1_{Y^*_0})=1_{\mathcal{A}}$.
\item For any finite set $I$, one has $f(\sum_{i\in I}\alpha_iw_i)=\sum_{i\in I}\alpha_if(w_i)$ where $\sum_{i\in I}\alpha_iw_i \in \mathbb{P}_n$( finite non-trivial positive linear combination).
\end{enumerate}
Then $f$ can be uniquely extended as a character 
i.e. $S_f=\sum_{w\in Y^*_0}f(w) w$ is group-like for $\Delta_{\shuffle}$.
\end{lemma}
\begin{proof}
The linear span of $\mathbb{P}_n$ is the space of homogeneous polynomials of degree $n$, ({\it i.e.,} $\mathbb{P}_n-\mathbb{P}_n=\ncp{\R_n}{Y_0})$, $\mathbb{P}_n$  being convex (and non-void), $f$ extends uniquely,    as a linear map, to $\ncp{\R_n}{Y_0}$ and then, as a linear map, on $\oplus_{n\geq 0} \ncp{\R_n}{Y_0}=\ncp{\R}{Y_0}$. This linear extension is a morphism for the shuffle product as it is so on the (linear) generators $Y_0^*$. 

By definition of $f$ and $S_f$, it is immediate $\scal{S_f}{1_{Y^*_0}}=1_{\mathcal{A}}$. One can check easily that $\Delta_{\shuffle}(S_f)=S_f\otimes S_f$. Hence, $S_f$ is group-like, for $\Delta_{\shuffle}$.
\end{proof}

\begin{corollary}
The noncommutative generating series $C^-$ is group-like, for $\Delta_{\shuffle}$.
\end{corollary}
\begin{proof}
It is a consequence of Lemma \ref{HGM} and Corollary \ref{tac}.
\end{proof}
\begin{example}[of $C^-_{u\shuffle v}$ and $C^-_{u\stuffle v}$]~\newline
{
\begin{tabular}{|c|c|c|c|c|c|}
\hline
$u$ &$C^-_u$&$v$ &$C^-_v$ &$u\shuffle v$& $C^-_{u\shuffle v}$ \\ 
\hline
$y_0$&$1$ & $y_0$&$1$ &$2y^2_0$& $1$ \\ 
 \hline
 $y_0^2$&$\frac{1}{2}$ & & &&  \\ 
 \hline
$y_1$&$\frac{1}{2}$ & $y_2$&$\frac{1}{3}$ &$y_1y_2 +y_2y_1$& $\frac{1}{6}$ \\ 
 \hline
$y_1y_2$&$\frac{1}{15}$ & $y_2y_1$&$\frac{1}{10}$ && \\ 
 \hline
$y_m$&$\frac{1}{m+1}$ & $y_n$&$\frac{1}{n+1}$ &$y_my_n +y_ny_m$& $\frac{1}{(m+1)(n+1)}$ \\ 
 \hline
$y_my_n$&$\frac{(n+1)^{-1}}{(n+m+2)}$ & $y_ny_m$&$\frac{(m+1)^{-1}}{(m+n+2)}$ && \\ 
 \hline
$y_1$&$\frac{1}{2}$& $y_2y_5$ &$\frac{1}{54}$& $y_1y_2y_5 + y_2y_1y_5 + y_2y_5y_1$& $\frac{1}{108}$\\
\hline 
$y_1y_2y_5$&$\frac{1}{594}$ & $y_2y_1y_5$&$\frac{1}{528}$ && \\ 
 \hline
$y_2y_5y_1$&$\frac{1}{176}$ & & && \\ 
 \hline
$y_0y_1$&$\frac{1}{6}$ & $y_2y_3$&$\frac{1}{28}$ &$y_0y_1y_2y_3 + y_0y_2y_1y_3$ & $\frac{1}{168}$ \\ 
& & & &$+ y_0y_2y_3y_1 + y_2y_3y_0y_1$&  \\ 
& & & &$+ y_2y_0y_1y_3 + y_2y_0y_3y_1$&  \\ 
 \hline
$y_0y_1y_2y_3$&$\frac{1}{2520}$ & $y_0y_2y_1y_3$&$\frac{1}{2160}$ && \\ 
 \hline
$y_0y_2y_3y_1$&$\frac{1}{1080}$ & $y_2y_3y_0y_1$&$\frac{1}{420}$ && \\ 
 \hline
$y_2y_0y_1y_3$&$\frac{1}{1680}$ & $y_2y_0y_3y_1$&$\frac{1}{840}$ && \\ 
 \hline
$y_ay_b$&$\frac{(b+1)^{-1}}{(a+b+2)}$ & $y_cy_d$&$\frac{(d+1)^{-1}}{(c+d+2)}$ &$y_ay_by_cy_d + y_ay_cy_by_d$ & $\frac{(b+1)^{-1}(d+1)^{-1}}{(a+b+2)(c+d+2)}$ \\ 
& & & &$+ y_ay_cy_dy_b + y_cy_dy_ay_b$&  \\ 
& & & &$+ y_cy_ay_by_d + y_cy_ay_dy_b$&  \\ 
 \hline
\end{tabular}\\[3pt]
\noindent
\begin{tabular}{|c|c|c|c|c|c|}
\hline
$u$ &$C^-_u$&$v$ &$C^-_v$ &$u\stuffle v$& $C^-_{u\stuffle v}$ \\ 
\hline
$y_0$&$1$ & $y_0$&$1$ &$2y^2_0 + y_0$& $1$ \\ 
 \hline
$y_1$&$\frac{1}{2}$ & $y_2$&$\frac{1}{3}$ &$y_1y_2 +y_2y_1 + y_3$& $\frac{1}{6}$ \\ 
 \hline
$y_m$&$\frac{1}{m+1}$ & $y_n$&$\frac{1}{n+1}$ &$y_my_n +y_ny_m + y_{n+m}$& $\frac{1}{(m+1)(n+1)}$ \\ 
 \hline
$y_1$&$\frac{1}{2}$& $y_2y_5$ &$\frac{1}{54}$& $y_1y_2y_5 + y_2y_1y_5 + y_2y_5y_1$& $\frac{1}{108}$\\ 
& & & &$+y_3y_5 + y_2y_6$& \\ 
 \hline
$y_0y_1$&$\frac{1}{6}$ & $y_2y_3$&$\frac{1}{28}$ &$y_0y_1y_2y_3 + y_0y_2y_1y_3$ & $\frac{1}{168}$ \\ 
& & & &$+ y_0y_2y_3y_1 + y_2y_3y_0y_1$&  \\ 
& & & &$+ y_2y_0y_1y_3 + y_2y_0y_3y_1 + y_0y_2y_4$&  \\ 
& & & & $+ y_0y_3^2+ y_2y_3y_1 + y_2y_1y_3$&  \\ 
& & & & $+ y_2y_0y_4 + y_2y_3y_1 + y_2y_4$&  \\
 \hline
$y_ay_b$&$\frac{(b+1)^{-1}}{(a+b+2)}$ & $y_cy_d$&$\frac{(d+1)^{-1}}{(c+d+2)}$ &$y_ay_by_cy_d + y_ay_cy_by_d$ & $\frac{(b+1)^{-1}(d+1)^{-1}}{(a+b+2)(c+d+2)}$ \\ 
& & & &$+ y_ay_cy_dy_b + y_cy_dy_ay_b$&  \\ 
& & & &$+ y_cy_ay_by_d + y_cy_ay_dy_b$&  \\ 
& & & &$+ y_ay_cy_{b+d} + y_ay_{b+c}y_d$&  \\ 
& & & &$+ y_cy_ay_{b+d} + y_cy_{a+d}y_b$&  \\ 
& & & &$+ y_{a+c}y_by_d + y_{a+c}y_dy_b + y_{a+c}y_{b+d}$&  \\
\hline
\end{tabular}
}

In the above tables, it can be clearly seen that 
$C^-_{\bullet}$ is linear on $\mathbb{P}_n$.
For example, let $u = y_1$ and $v = y_2y_5$. Then $u \shuffle v=y_1y_2y_5+y_2y_1y_5+y_2y_5y_1$.
Hence, we get 
\begin{eqnarray*}
C^-_{y_1y_2y_5}+C^-_{ y_2y_1y_5}+C^-_{ y_2y_5y_1}=\frac{1}{594}+\frac{1}{528}+\frac{1}{176}=\frac{1}{108}
=C^-_{y_1}C^-_{ y_2y_5}=C^-_{y_1 \shuffle y_2y_5}.
\end{eqnarray*}
Note that $y_1y_2y_5,y_2y_1y_5,y_2y_5y_1\in\mathbb{P}_{11}$.
But we have also $u \stuffle v = y_1y_2y_5 + y_2y_1y_5 + y_2y_5y_1 +y_3y_5 + y_2y_6$. It is clearly seen that
\begin{eqnarray*}
C^-_{y_1y_2y_5}+ C^-_{ y_2y_1y_5}+ C^-_{ y_2y_5y_1}+C^-_{y_3y_5}+C^-_{y_2y_6}=\frac{1}{108}+\frac{13}{420}\neq \frac{1}{108}=C^-_{y_1} C^-_{y_2y_5}.
\end{eqnarray*}
However, as $y_3y_5,y_2y_6 \in\mathbb{P}_{10}$, we can conclude that
\begin{eqnarray*}
C^-_{y_1 \stuffle y_2y_5} = C^-_{y_1y_2y_5 + y_2y_1y_5 + y_2y_5y_1 +y_3y_5 + y_2y_6} = C^-_{y_1y_2y_5 + y_2y_1y_5 + y_2y_5y_1} = {1}/{108} = C^-_{y_1}C^-_{y_2y_5}.
\end{eqnarray*}
\end{example}

\subsection{Structure of polylogarithms at non-positive multi-indices}\label{sectiontop}
We are constructing a new product on $\Q\left\langle{Y_0}\right\rangle$ which is associated with polylogarithms at non-positive multi-indices as an algebra. This construction will rest on the following very general 
lemma which gives a way to find a preimage of the given law.
\begin{lemma}\label{law}
Let $k$ be a field, $V$ a $k$-vector space, $\A$ a $k$-AAU and  $\phi:V\twoheadrightarrow\A$
be an onto (linear) mapping. We consider the $k$-AA\footnote{$k$- associative algebra (not necessarily with unit); $k$- AAU being for the category of $k$- associative algebra with unit.} laws on $V$ such that, for all $x,y\in V$ 
\begin{eqnarray}
&\phi(x\top y)=\phi(x)\phi(y).\label{eq105}\\
&\mbox{If}\quad x\top y\in\ker(\phi)\quad\mbox{then}\quad x\top y=0.\label{eq205}
\end{eqnarray}
Then the following clauses hold
\begin{enumerate}
\item There is at least a solution $\top_0$ of (\ref{eq105}) and (\ref{eq205}).
\item Let $G_\phi$ be the group of (linear) automorphisms (subgroup of $GL(V)$) stabilizing $\phi$ on the right, {\it i.e.} $G_\phi:=\{\alpha\in GL(V)|\phi\circ\alpha=\phi\}$.
Then any other law $\top$ satisfying (\ref{eq105}) and (\ref{eq205}) is of the form $\top=\alpha\circ\top_0$ (one orbit under $G_\phi$).
\item This statement holds with  $G^{(1)}_\phi:=\{\alpha\in GL(V)|\phi\circ\alpha=\phi;\alpha_{|_{\ker(\phi)}}=\mathrm{Id}_{\ker(\phi)}\}$.
\end{enumerate}
\end{lemma}

\begin{proof}
\begin{enumerate}
\item Let $s$ be a linear section of $\phi$, {\it i.e.} $\phi\circ s=\mathrm{Id}_\A$. Then let  $x\top y=s(\phi(x)\phi(y))$.
It is straightforward $\top$ satisfies (\ref{eq105}). Now, $x\top y=s(\phi(x)\phi(y))\in\mathrm{Im}(s)$ if, moreover $x\top y\in\ker(\phi)$,
as $(0)=\mathrm{Im}(s)\cap\ker(\phi)$,  
we must have $x\top y=0$ and hence (\ref{eq205}).
\item We have to prove ``Every product $\top$ of the form $\top=\alpha\circ\top_0$ satisfies (\ref{eq105})-(\ref{eq205})''
and conversely ``every product $\top$ which satisfies (\ref{eq105})-(\ref{eq205}) is of the form $\top=\alpha\circ\top_0$''.

Firstly, let us compute 
$\phi(x\top y)=\phi(\alpha(x\top_0 y))=\phi(x\top_0 y)=\phi(x)\phi(y)$.
This proves (\ref{eq105}). On the other hand, $x\top y\in \ker(\phi)$ is equivalent to $\phi(x\top y)=0$.
Then $0=\phi(x\top y)=\phi(\alpha(x\top_0 y))=\phi(x\top_0 y)$.
Hence, $x\top_0 y=0$ and, because $\top_0$ satisfies (\ref{eq205}), and then, $x\top y=\alpha(x\top_0 y)=0$ which proves (\ref{eq205}) for $\top$.

Now, let us suppose that $\top$ satisfies (\ref{eq105}) and (\ref{eq205}) and first compute an idempotent (for $\top$) $e$ such that $\phi(e)=1_\A$.
We start with any preimage $e_0$ of $1_\A$ and form $e=e_0\top e_0$. Then 
$\phi(e_0\top e_0-e_0\top e_0\top e_0)=\phi(e_0)\phi(e_0)-\phi(e_0)\phi(e_0)\phi(e_0)=1-1=0.$
Hence, $e_0\top e_0-e_0\top e_0\top e_0=e_0\top(e_0-e_0\top e_0)\in \ker(\phi)$ and then, by (\ref{eq205}), we have
$0=e_0\top(e_0-e_0\top e_0)=e_0\top e_0-e_0\top e_0\top e_0$.

Now $e\top e=(e_0\top e_0)\top (e_0\top e_0)=e_0\top (e_0\top e_0\top e_0)=e_0\top (e_0\top e_0)=e_0\top e_0=e$  
and $\phi(e)=\phi(e_0\top e_0)=1_\A1_\A=1_\A$.

Now, let $y\in\A$. It is easy to check that the values $x\top e$ are independant from the choice of $x$, preimage of $y$.
Set $s(y)=x\top e$.  For $y=\phi(x)$, we have
\begin{eqnarray}\label{s_def}
s(y)=s(\phi(x))=x\top e.
\end{eqnarray} 
The map $s$ is a section of $\phi$ as $\phi(s(y))=\phi(x\top e)=\phi(x)\phi(e)=\phi(x)$.

To end, it remains to prove that $x\top y=s(\phi(x)\phi(y))$ but
\begin{eqnarray}\label{tbf1}
s(\phi(x)\phi(y))=s(\phi(x\top y))=x\top y\top e.
\end{eqnarray}
Now $x\top y\top e-x\top y$ has image, by $\phi$, zero and then $x\top(y\top e-y)=0$ which, in virtue of (\ref{tbf1}),
ends the proof that $s$ exists and is such that $x\top y=s(\phi(x)\phi(y))$.

Now, if $s_1,s_2$ are two sections of $\phi$, there is $\alpha\in G^{(1)}_\phi$ such that
$s_2=\alpha s_1$. We can reprove it easily in our context. One has just to consider a basis of $\mathrm{Im}(s_1)$, take into account that 
$\ker(\phi)\oplus\mathrm{Im}(s_1)=\ker(\phi)\oplus\mathrm{Im}(s_2)=V$
and construct $\alpha$ by $\alpha_{|_{\mathrm{Im}(s_1})}=s_2,\phi_{|_{\mathrm{Im}(s_1)}}$ and $\alpha_{|_{\ker(\phi})}=\mathrm{Id}_{\ker(\phi)}$.
Let $\top_i$ be constructed from $s_i,i=1,2$. Then $x\top_2 y=s_2(\phi(x)\phi(y))=\alpha s_1(\phi(x)\phi(y))=\alpha(x\top_1 y)$.

\item Straightforward following the proof of the previous point.
\end{enumerate}
\end{proof}

\begin{theorem}\label{onto}
\begin{enumerate}
\item There is at least a law $\top$ such that the morphism defined in (\ref{maps2}) is onto and, such that, for any $P,Q\in\ncp{\Q}{Y_0}$,
\begin{eqnarray}
&\Li^-_{P\top Q}=\Li^-_P\Li^-_Q.\label{eq01}\\
&\mbox{If}\quad P\top Q\in \ker(\Li^-_{\bullet})\quad\mbox{then}\quad P\top Q=0.\label{eq02}
\end{eqnarray}
Then, let $G = \{ \alpha\in GL(\ncp{\Q}{Y_0})|\Li^-_{\bullet}\circ\alpha= \Li^-_{\bullet}\}\subset GL(\ncp{\Q}{Y_0}).$
\item Any other law satisfying \eqref{eq01}, \eqref{eq02} is of the form $\alpha\circ\top$( one orbit under $G$).
\item The associative commutative law $\top$ in \eqref{definitiontop} is non dualizable.
\end{enumerate}
\end{theorem}

\begin{proof}
For $V=\ncp{\Q}{Y_0}$ and $\phi=\Li^-_{\bullet}$,  the first result is a consequence of Lemma \ref{law}.
The last is a consequence of the caracterization of dualizability \cite{BDHHT}.
\end{proof}

By Theorem \ref{pr2}, $\{\Li^-_w\}_{w\in Y^*_0}$ can be represented in $\{\Li^-_{y_s}\}_{s\ge0}\cup\{\Li^-_{1_{Y^*_0}}\}$.
Thus, since, for any $u,v \in Y^*_0$, $\Li^-_u\Li^-_v\in\mathrm{span}_{\Q}{\{\Li^-_{y_k}\}_{k\ge0}}$ then we get successively
\begin{eqnarray*}
\Li^-_u\Li_v^-&=&a_{1_{Y^*_0}}(u,v)+\Sum_{s = 0}^{|u|+(u)+|v|+(v)-1}a_s(u,v)\Li^-_{y_s},
\end{eqnarray*}
and then, one defines
\begin{eqnarray}\label{definitiontop}
u \top v&=&a_{1_{Y^*_0}}(u,v)1_{Y^*_0}+\Sum_{s=0}^{|u|+(u)+|v|+(v)-1}a_s(u, v) y_s.
\end{eqnarray}
This defines a law of $\top$ which is associative and commutative on $\Q\left\langle{Y_0}\right\rangle$.
\begin{example}\label{Top}
\begin{enumerate}
\item Since $\Li^-_{1_{Y^*_0}}=\1_{\calC}$ then, for  $y\in Y_0$, $y\top 1_{Y^*_0}=1_{Y^*_0}\top y=y$.
\item One has\footnote{For any $m,n,k \in \N$ such that $0\le k\le m+n-2$, we denote $A_{m,n,k}=\sum_{t=0}^{k}A_{n,t}A_{m, k-t}$.
Here, for any $m, n\in\N$, $A_{m,n,-1}=A_{m,n, -2}=0$.},
\begin{eqnarray*}
\Li^-_{y_m}(z) \Li^-_{y_n}(z)&=&\frac{z^2}{(1-z)^2}\frac{1}{(1-z)^{m+n}}\sum_{k=0}^{m+n-2}A^-_{m,n,k}z^k\cr
&=&\Sum_{k=0}^{m+n-2} A_{m,n,k} \Sum_{t=0}^{k+2} \binom{k+2}{t} \Frac{(-1)^t}{(1-z)^{m+n+2-t}} \cr
&=&\Sum_{k=2}^{m+n+2}\Sum_{j=m+n-k}^{m+n-2} A_{m,n,j} \binom{j+2}{m+n+2-t} (-1)^{m+n-k}\Frac{1}{(1-z)^k}.
\end{eqnarray*}
By Theorem \ref{pr2} and since $(1-z)^{-1}=\Li^-_{y_0}(z)-1_{\calC}$ then we obtain successively
\begin{eqnarray*}\label{multipleLi}
\Li^-_{y_m}(z) \Li^-_{y_n}(z)&=&\Sum_{k=2}^{m+n+2}\gamma_{m,n,k}\bigg(\Frac{1}{1-z}
+\Sum_{j=2}^{k}\Frac{S_1(k,j)}{k!}\Li^-_{y_{j-1}}(z)\biggr),\\
y_m \top y_n&=&\Sum_{k=2}^{m+n+2}\gamma_{m,n,k} \bigg(y_0-1_{Y^*_0}
+\Sum_{j=2}^{k}\Frac{S_1(k,j)}{k!}y_{j-1}\biggr),
\end{eqnarray*}
where $\gamma_{m,n,k}:=\sum_{j=m+n-k}^{m+n-2}A_{m,n,j} \binom{j+2}{m+n+2-t}(-1)^{m+n-k}$.
For example,
$$\begin{array}{rcl}
y_5 \top y_4 &= & -\frac{1}{60}y_2 + \frac{1}{63}y_4 + \frac{1}{1260}y_{10},\cr
y_5 \top y_5 &= & -\frac{5}{66}y_1 + \frac{1}{12}y_3 -\frac{1}{126}y_5 + \frac{1}{2772}y_{11},\cr
y_6 \top y_7 &=& - \frac{691}{5460} y_2 + \frac{5}{33}y_4 - \frac{1}{40}y_6 + \frac{1}{24024}y_{14},\cr
y_8 \top y_{10} &=& \frac{43867}{798}y_1 - \frac{39787}{510}y_3 + \frac{77}{3}y_5 -\frac{11056}{4095}y_7 + \frac{5}{66}y_9 + \frac{1}{831402}y_{19}. 
\end{array}$$
\end{enumerate}
\end{example}

\begin{corollary}
\begin{enumerate}
\item For any $P,Q\in\Q\left\langle{Y}\right\rangle\oplus\Q y_0 $, one has  $B^-_PB^-_Q=B^-_{P\top Q }$.
\item Let $u,v\in Y_0^*$. Then $\scal{u\top v}{1_{Y^*_0}}\ne0$ if and only if $u=v=1_{Y^*_0}$.
\item For any $u,v$ and $w=y_{s_1}\ldots y_{s_r}\in Y^*_0$, one has $y_0u\top v=u \top y_0v=y_0\top(u\top v)$
and $w\top1_{Y^*_0}=\sum_{i=0}^{s_1}\binom{s_1}{i}y_i\top y_{s_1+s_2-i}y_{s_3}\ldots y_{s_r}$.
\item The map of (\ref{maps}) is onto and
$\ker\H^-_{\bullet}=\ker\Li^-_{\bullet}=\Q\langle{\{w-w\top1_{Y_0^*}|w\in Y_0^*\}}\rangle$.
\end{enumerate}
\end{corollary}
\begin{proof}
\begin{enumerate}
\item Let $P,Q\in\Q\left\langle{Y}\right\rangle\oplus\Q y_0$ such that
$\Li^-_P=a^P_{k+1}\Li^-_{y_k}+\ldots+a^P_01_{\calC}$ and $\Li^-_Q=a^Q_{l+1}\Li^-_{y_l}+\ldots+a^Q_01_{\calC}$.
Then $\Li^-_P\Li^-_Q
=a^P_{k+1}a^Q_{l+1}\Li^-_{y_l}\Li^-_{y_k}+\ldots+a^P_0a^Q_01_{\calC}
=a^P_{k+1}a^Q_{l+1}\Li^-_{y_l\top y_k}+\ldots+a^P_0a^Q_01_{\calC}$.
Since $B^-_P=a_{k+1}^PB^-_{y_k}$ and  $B^-_P=a_{l+1}^QB^-_{y_l}$
then $B^-_{P \top Q} =a_{k+1}^Pa_{l+1}^QB^-_{y_k \top y_l}$.
Finally, $B^-_{P\top Q}=a_{k+1}^P a_{l+1}^QB^-_{y_k}B^-_{y_l}=B^-_PB^-_Q.$
\item For $a_{1_{Y^*_0}}(u,v)\in\Q^*$ and $a_s(u,v)\in\Q$, the law $\top$ given in (\ref{definitiontop}) yields
$$u\top v = a_{1_{Y^*_0}}(u,v)1_{Y^*_0}+\sum_{|u|+(u)+|v|+(v)-1 \geq s \geq 0}a_s(u,v)y_s,$$
{\it i.e.} $\Li^-_u\Li^-_v = a_{1_{Y^*_0}}(u,v)+\sum_{|u|+(u)+|v|+(v)-1 \geq s \geq 0}a_s(u,v)\Li^-_{y_s}.$

Suppose $u\neq1_{Y^*_0}$, since $(1-z)^{|u|+(u)}\Li^-_u(z)=z^{|u|}A_u(z)$ then
multiplying both two sides by $(1 - z)^{|u|+(u)+|v|+(v)}$, one has
\begin{eqnarray*}\label{triettieu}
z^{|u|}A_u(z)(1-z)^{|v|+(v)}\Li^-_v(z)=(1-z)^{|u|+(u)+|v|+(v)}a_{1_{Y^*_0}}(u,v)\\
+(1-z)^{|u|+(u)+|v|+(v)}\sum_{s=0}^{|u|+(u)+|v|+(v)-1}a_{s}(u, v)\Li^-_{y_s}(z).
\end{eqnarray*}
Note that $(1-z)^{|u|+(u)+|v|+(v)}\Li^-_{y_s}(z)$ vanishes at $z=0$,
for any $s =1,\ldots,|u|+(u)+|v|+(v)-1$ and $(1-z)^{(v)+|v|}\Li^-_v(z)\in\Q[(1-z)^{-1}]$.
For $z=0$, it is equivalent to $0=a_{1_{Y^*_0}}(u,v)$ contradicting with the assumptions. Hence, $u=1_{Y^*_0}$.

Similarly, we also obtain $v=1_{Y^*_0}$, then the expected result follows.

\item The formulas \eqref{noyaux}-\eqref{remark}  lead to $\Li^-_{y_0u}=(\theta_0 \iota_1) \Li^-_u=\lambda\Li^-_u=\Li^-_{y_0}\Li^-_u$.
Hence, $\Li^-_{y_0 u \top v}=\Li^-_{y_0} \Li^-_{u \top v}$, {\it i.e.} $y_0u \top v=y_0 \top (u \top v)$.
The next result is a consequence of the definition of $\top$ and of Theorem \ref{pr2}.

\item It is immediate by Corollary \ref{onto}.
\end{enumerate}
\end{proof}

\subsection{Natural bases for  harmonic sums and polylogarithms at non-positive multi-indices}\label{translate}
We are building a map allowing to represent  
$\{\H^-_w\}_{w\in Y_0^*}$ in terms of the basis $(N^k)_{k\geq 0}$
from the representation of the corresponding $\{\Li^-_w\}_{w\in Y_0^*}$ in the basis $\{(1-z)^{-k} \}_{k \geq 0}$.

\begin{lemma}\label{solutions_diff}
Consider the difference equation 
\begin{equation}\label{diffeq}
f(x+1)-f(x)=P(x),
\end{equation} 
where $P$ is a polynomial with coefficients in the $\Q$-algebra $\A$ and $f:\N\longrightarrow A$ is an unknown function. Then
\begin{enumerate}
\item Let $\sum_{j=0}^d a_j\binom{x}{j}$ be the decomposition of $P$ w.r.t. the basis $\left\{\binom{x}{k}\right\}_{k\ge0}$.
Then the unique solution without constant term of \eqref{diffeq} is $f_0(x)=\sum_{j=0}^{d} a_j\binom{x}{j+1}$.
\item All solutions of the equation \eqref{diffeq} are polynomial functions and differ by a constant.   
\item If $P\not\equiv 0$ with leading term $a_d x^d$,
any solution of the difference equation has leading term $a_d{x^{d+1}}/(d+1)$.
\end{enumerate}
\end{lemma}

\begin{lemma}\label{combination}
For any $n,N \in \N_+$,  one has
\begin{eqnarray*}\label{t}
\binom{N}{n}=\Frac{1}{n!}\Sum_{t=1}^{n} (-1)^{n+t} S_1(n, t)N^t\iff N^n=\Sum_{j=1}^n j!S_2(n,j)\binom{N}{j}.
\end{eqnarray*}
Moreover, for any $i \geq j \in \N_+$, one has 
\begin{eqnarray*}
&&S_2(i,j)=\\
&&\left\{\begin{array}{rcl}
\Frac{1}{S_1(i,i)},\mbox{ if} \quad i=j,\\
\Frac{1}{S_1(i,i)S_1(j,j)}\Sum_{k=1}^{i -j}(-1)^{k+1}\Sum_{i>t_1>\ldots>t_k>j}
\Frac{(-1)^{i+j+ t_1+\ldots +t_k}S_1(i,t_1)\ldots S_1(t_k,j)}{S_1(t_1,t_1)\ldots S_1(t_k,t_k)}\\
-\Frac{S_1(i,j)}{S_1(i,i)S_1(j,j)}, \mbox{ if} \quad i > j.
\end{array}\right.
\end{eqnarray*}
\end{lemma}

\begin{proof}
In this paper, for any $i, j\in\N$ such that $0 \le i< j$, it follows from the definitions that $S_1(i,j)=S_2(i,j)=0$, and we can write the above formula in the matrix form
\begin{eqnarray*}
\begin{pmatrix}\binom{N}{k}\end{pmatrix}_{k\geq 1}^t
=\begin{pmatrix}1/{1!}&0&\ldots\\0&\1/2!&\ddots\\ \vdots&\ddots&\ddots\end{pmatrix}
\begin{pmatrix}(-1)^{i+j}S_1(i,j)\end{pmatrix}_{i,j\ge1}\begin{pmatrix}N^k\end{pmatrix}_{k\ge1}^t.
\end{eqnarray*}
Thus, by matrix inversion and the Stirling transform \cite{BS}, we get the first result~:  
\begin{eqnarray*}
\begin{pmatrix}N^k\end{pmatrix}_{k\geq 1}^t
=\begin{pmatrix}j!S_2(i,j)\end{pmatrix}_{i,j\geq 1}\begin{pmatrix}\binom{N}{k}\end{pmatrix}_{k\ge1}^t.
\end{eqnarray*}
The inverse matrix of $\begin{pmatrix}S_1(i,j)\end{pmatrix}_{i,j \geq 1}$, is well-known and leads also to the last result.
\end{proof}
\begin{theorem}\label{Tonghop}
\begin{enumerate}
\item Let $X=(j!S_2(i,j))_{j\ge1}^{i\ge1}$. One has
\begin{eqnarray*}
\begin{pmatrix}\H^-_{y_i}(N)\end{pmatrix}^t_{i\ge0}
=M\begin{pmatrix}N^j\end{pmatrix}_{j\ge1}^t
=MX\begin{pmatrix}\binom{N}{k}\end{pmatrix}_{k\ge1}
=MX\begin{pmatrix}\H^-_{y_0^k}(N)\end{pmatrix}_{k\ge1}.
\end{eqnarray*}
\item The families
$\{\H^-_{1_{Y^*_0}}\}\cup\{\H^-_{y_s}\}_{s\geq 0}$ and $\{\H^-_{1_{Y^*_0}}\}\cup\{\H^-_{y_0^k}\}_{k\ge0}$
(resp. $\{1_{\calC}\}\cup\{\Li^-_{y_s}\}_{s\geq 0}$ and $\{1_{\calC}\}\cup\{\Li^-_{y_0^k}\}_{k\ge0}$)
are bases of $\Q[\{\H^-_w\}_{w \in Y^*_0}]$ (resp. $\Q[\{\Li^-_w\}_{w \in Y^*_0}]$).

Moreover, $\scal{\H^-_w}{\H^-_{1_{Y^*_0}}}\neq0$ if and only if $w=1_{Y^*_0}$. 
\item The automorphism $\chi$ is represented by the matrix $\begin{pmatrix}1&0\\ M^{-1}U&M^{-1}T\end{pmatrix}^t=T^t$,
on the basis $\{X^k\}_{k\ge0}$ of $\Q[X]$, and, for $w\in Y_0^*$, $\Li^-_w(z)=(\chi\circ\H^-_w)((1-z)^{-1})$. 

Moreover, by Theorems \ref{degreeofHarmonic} and \ref{pr2}, one has
\begin{eqnarray*}
\H^-_{w}(N)=\sum_{k=0}^{(w)+|w|}m_kN^k\in\Q[N]\iff\Li^-_w(z)=\sum_{k=0}^{(w)+|w|}\frac{n_k}{(1-z)^k}\in\Z[(1-z)^{-1}],
\end{eqnarray*}
where, for any $0\le k\le(w)+|w|$, $n_k=\sum_{j=k}^{(w)+|w|}m_jt_{j,k}$.
\end{enumerate}
\end{theorem}

\begin{proof}
\begin{enumerate}
\item Since, for any $n,N\in\N_+$, $\binom{N}{n}=\H^-_{y^n_0}(N)$ then, by 
the definition of $M$ as (\ref{M}) and by Lemma \ref{combination}, 
the expression of the matrix $X$  follows.

\item For any  $w\in Y^+_0$, there is a rational sequence $\{\alpha_{w,k}\}_{k=0}^{(w)+|w|-1}$ such that
\begin{eqnarray*}
\H^-_w=\sum_{k=0}^{(w)+|w|-1}\alpha_{w, k}\H^-_{y_k}\iff\Li^-_w=\sum_{k=0}^{(w)+|w|-1}\alpha_{w, k}\Li^-_{y_k}.
\end{eqnarray*}
Indeed, by
$$\frac{\Li^-_w(z)}{1-z}
=\sum\limits_{N \ge0}\biggl(\sum_{k=0}^{(w)+|w|-1}\alpha_{w, k}\H^-_{y_k}(N)\biggr)z^N
=\sum\limits_{k=0}^{(w)+|w|-1}\alpha_{w, k}\biggl(\sum_{N \geq 0}\H^-_{y_k}(N)\biggr)z^N$$
then Theorems \ref{degreeofHarmonic}, \ref{pr2} and the previous point implies that these families are bases.

\item Let us define, for any $w \in Y^*_0$, 
$h^-_w:=\begin{pmatrix}\H^-_{y_kw}\end{pmatrix}_{k\ge0}^t$ and $l^-_w:=\begin{pmatrix}\Li^-_{y_kw}\end{pmatrix}_{k\ge0}^t$. Let $U:=\begin{pmatrix}-1\cr0\end{pmatrix}\in Mat_{\infty}(\Q)$. With the matrices $M,T$ given respectively in \eqref{T}, \eqref{M}
and by Theorems \ref{degreeofHarmonic}, \ref{pr2}, there exists two matrices $\Xi_w,\Omega_w\in Mat_{\infty}(\Q)$ such that
\begin{eqnarray*}\label{matrices_debut}
h^-_w=&\Xi_w\begin{pmatrix}1\cr \begin{pmatrix}N^j\end{pmatrix}_{j\ge1}^t\end{pmatrix}
&=\Xi_w\begin{pmatrix}1&0\\ 0&M\end{pmatrix}^{-1}\begin{pmatrix}1\\ h^-_{1_{Y^*_0}}\end{pmatrix},\\
l^-_w=&\Omega_w\begin{pmatrix}\1_{\calC}\cr \begin{pmatrix}(1-z)^{-j}\end{pmatrix}_{j\ge1}^t\end{pmatrix}
&=\Omega_w\begin{pmatrix}1&0\\ U&T^{-1}\end{pmatrix}^{-1}\begin{pmatrix}\1_{\calC}\\ l^-_{1_{Y^*_0}}\end{pmatrix},
\end{eqnarray*}
 Then one has successively
\begin{eqnarray*}
\Xi_w\begin{pmatrix}1&0\\0&M\end{pmatrix}^{-1}=\Omega_w\begin{pmatrix}1&0\\ U&T^{-1}\end{pmatrix}^{-1}
&\mbox{and}&
\Xi_w=\Omega_w\begin{pmatrix}1&0\\ U&T^{-1}\end{pmatrix}^{-1}\begin{pmatrix}1&0\\ 0&M\end{pmatrix},\\
\begin{pmatrix}1\cr h^-_{1_{Y^*_0}}\end{pmatrix}=\begin{pmatrix}1&0\cr0&M\end{pmatrix}
\begin{pmatrix}1\\ \begin{pmatrix}N^j\end{pmatrix}_{j\ge1}^t\end{pmatrix}
&\mbox{and}&
\begin{pmatrix}\1_{\calC}\\ l^-_{1_{Y^*_0}}\end{pmatrix}
=\begin{pmatrix}1&0\\ U&T^{-1}\end{pmatrix}\begin{pmatrix}\1_{\calC}\\ \begin{pmatrix}\frac1{(1-z)^j}\end{pmatrix}_{j\ge1}^t\end{pmatrix}.\label{matrices_fin}
\end{eqnarray*}
The expected result follows. 
\end{enumerate}
\end{proof}

\section{Conclusion}
In this work, we have etablished combinatorial and asymptotic aspects concerning harmonic sums and polylogarithms at non-positive multi-indices, by extending Faulhaber's formula, the Bernoulli and Eulerian polynomials.

Via an Abel like theorem about their noncommutative generating series, we have also globally renormalized
the corresponding polyzetas and made precise their algebraic structures.

In the forthcoming works, we will give an integral representation for polylogarithms at non-positive multi-indices,
$\{\Li^-_w\}_{w\in Y^*_0}$, for regularization of the corresponding polyzetas.

\section{Appendix A~: Cones and extensions}\label{appendix2}
Let $V$ be a $\R$-vector space.
We remind \cite{Dennis} the reader that a {\it blunt} convex cone in $V$ is a convex cone which does not contain zero.
If $C\not=\emptyset$ is such a cone, then the vector space generated by $C$ is
\begin{eqnarray}\label{differences}
\mathrm{span}_\R(C)=C-C=\{x-y\}_{x,y\in C}.
\end{eqnarray}

Let $S\subset X^*$ a non empty linearly free set. The blunt convex cone generated by $S$ is the set of sums
$C_S=\{\sum_{w\in S}\alpha_w w\}_{\alpha \in \R_+^{(S)}\setminus\{0\}}$.
It amounts to the same to rephrase it with finite families and,
in view of (\ref{differences}), the linear span of $S$ is exactly $C_S-C_S$.
 
Now, $C$ being still a non-empty convex cone, we say that a function $\varphi:C\rightarrow W$
($W$ be a $\R$-vector space) is linear on $C$ if $\phi(\alpha x+\beta y)=\alpha\phi(x)+\beta\phi(y)$
for $x,y\in C,\alpha,\beta\ge0,$ $\alpha+\beta>0$.
This is an easy exercise to check that such a $\varphi$ is the restriction of a unique linear map $\hat{\varphi}\ (C-C)\rightarrow W$.

\end{document}